\documentclass[12pt]{amsart}
\usepackage{amssymb}
\usepackage{tikz}
\usetikzlibrary{arrows,matrix,positioning,automata,shapes,calc}
\usetikzlibrary{decorations.pathreplacing,calligraphy}
\usepackage{mathtools}

\usepackage{nicematrix}
\NiceMatrixOptions{nullify-dots}

\usepackage{accents}
\newcommand{\dbtilde}[1]{\accentset{\approx}{#1}}

\let\frak\mathfrak
\let\Bbb\mathbb

\usepackage{eucal}

\usepackage{graphicx}

\def\>{\relax\ifmmode\mskip.666667\thinmuskip\relax\else\kern.111111em\fi}
\def\:{\relax\ifmmode\mskip.333333\thinmuskip\relax\else\kern.0555556em\fi}
\def\<{\relax\ifmmode\mskip-.333333\thinmuskip\relax\else\kern-.0555556em\fi}
\def\?{\relax\ifmmode\mskip-.666667\thinmuskip\relax\else\kern-.111111em\fi}
\def\vsk#1>{\vskip#1\baselineskip}
\def\vv#1>{\vadjust{\vsk#1>}\ignorespaces}
\def\vvn#1>{\vadjust{\nobreak\vsk#1>\nobreak}\ignorespaces}
 \let\alb\allowbreak

\def\plait#1{\par\hangindent2\parindent\indent\kern\parindent
	\llap{#1\enspace}\ignorespaces}

\let\Smallskip\smallskip
\def\smallskip{\par\Smallskip}
\let\Medskip\medskip
\def\medskip{\par\Medskip}
\let\Bigskip\bigskip
\def\bigskip{\par\Bigskip}

\let\Maketitle\maketitle
\def\maketitle{\Maketitle\thispagestyle{empty}\let\maketitle\empty}

\newtheorem{thm}{Theorem}[section]

\newtheorem{lem}[thm]{Lemma}
\newtheorem{prop}[thm]{Proposition}

\newtheorem{rem}[thm]{Remark}
\newtheorem*{rem*}{Remark}

\numberwithin{equation}{section}

\theoremstyle{definition}

\theoremstyle{definition}

\newtheorem*{example}{Example}

\def\beq{\begin{equation}}
\def\eeq{\end{equation}}
\def\be{\begin{equation*}}
	\def\ee{\end{equation*}}

\def\bean{\begin{eqnarray}}
\def\eean{\end{eqnarray}}
\def\bea{\begin{eqnarray*}}
	\def\eea{\end{eqnarray*}}

 \let\eps\varepsilon \let\epsilon\eps

 \let\phi\varphi

\let\geq\geqslant

\let\leq\leqslant

\def\C{\Bbb C}
\def\Z{\Bbb Z}

\def\gl{\frak{gl}}
\def\g{\frak{g}}

\def\lsym#1{#1\alb\dots\relax#1\alb} \def\lc{\lsym,}
\def\sq{\text{\scalebox{0.7}{$\square$}}}

\let\on\operatorname

\def\End{\on{End}}

\def\Hom{\on{Hom}}
\def\id{\on{id}}

\def\cdet{\on{cdet}}

\def\sign{\on{sign}}
\def\tr{\on{tr}}

\def\KZ/{{\sl KZ\/}}
\def\qKZ/{{\sl qKZ\/}}

\DeclareMathOperator*{\dashprod}{\scalebox{1.5}{$\Pi$}^{\prime}}

\makeatletter
\newcommand*{\shifttext}[2]{%
  \settowidth{\@tempdima}{#2}%
  \makebox[\@tempdima]{\hspace*{#1}#2}%
}
\makeatother

\newcommand\restr[2]{{
  \left.\kern-\nulldelimiterspace 
  #1 
  \right|_{#2} 
  }}

\makeatletter
\newcommand{\subalign}[1]{%
	\vcenter{%
		\Let@ \restore@math@cr \default@tag
		\baselineskip\fontdimen10 \scriptfont\tw@
		\advance\baselineskip\fontdimen12 \scriptfont\tw@
		\lineskip\thr@@\fontdimen8 \scriptfont\thr@@
		\lineskiplimit\lineskip
		\ialign{\hfil$\m@th\scriptstyle##$&$\m@th\scriptstyle{}##$\hfil\crcr
			#1\crcr
		}%
	}%
}
\makeatother

\title{Gaudin model and Deligne's category}

\author[B\<.\,Feigin]{B\<.\,Feigin$\:^\diamond$}
\thanks{$\kern-\parindent^\diamond$ E\:-mail: borfeigin@gmail.com}

\author[L\<.\,Rybnikov]{L\<.\,Rybnikov$\:^\circ$}
\thanks{$\kern-\parindent^\circ$ E\:-mail: leo.rybnikov@gmail.com}

\author[F\<.\,Uvarov]{F\<.\,Uvarov$\:^\star$}
\thanks{\noindent$^\star$ E\:-mail: fuvarov@hse.ru}

\begin{document}

\maketitle

\begin{center}
\vsk-.2>
{\it $^{\diamond\:\circ\:\star}\?$ HSE University, Faculty of Mathematics, \\ 
6 Usacheva str., Moscow, 119048, Russia\/}

\medskip
{\it $^\circ\?$ Harvard University, Department of Mathematics, \\ 1 Oxford Street, Cambridge, MA 02138, USA \/}

\medskip
{\it $^\diamond\?$ Hebrew University of Jerusalem, Einstein Institute of Mathematics, \\ Givat Ram. Jerusalem, 9190401, Israel  \/}
\end{center}

\begin{abstract}
    We show that the construction of the higher Gaudin Hamiltonians associated to the Lie algebra $\gl_{n}$ admits an interpolation to any complex $n$. We do this using the Deligne's category $\mathcal{D}_{t}$, which is a formal way to define the category of finite-dimensional representations of the group $GL_{n}$, when $n$ is not necessarily a natural number.
    
    We also obtain interpolations to any complex $n$ of the no-monodromy conditions on a space of differential operators of order $n$, which are considered to be a modern form of the Bethe ansatz equations. We prove that the relations in the algebra of higher Gaudin Hamiltonians for complex $n$ are generated by our interpolations of the no-monodromy conditions.  

    Our constructions allow us to define what it means for a pseudo-deifferential operator to have no monodromy. Motivated by the Bethe ansatz conjecture for the Gaudin model associated with the Lie superalgebra $\gl_{n\vert n'}$, we show that a ratio of monodromy-free differential operators is a pseudo-differential operator without monodromy.
\end{abstract}

\section{Introduction}
\subsection{} 
Symmetric tensor categories (STC's), see \cite{EGNO}, \cite{EK}, provide a natural framework for studying finite-dimensional representations $\on{Rep}_{\mathbb{K}}(G)$ of a group $G$ over an algebraically closed field $\mathbb{K}$. In this paper, we will consider only $\mathbb{K}=\C$. It is known that for any group $G$, there exists an affine group scheme $\widehat{G}$ such that the category $\on{Rep}_{\C}(G)$ is equivalent to the category $\on{Rep}^{alg}_{\C}(\widehat{G})$ of finite-dimensional (algebraic) representations of $\widehat{G}$. A natural question is whether there are examples of STC's which are not equivalent to $\on{Rep}^{alg}_{\C}(\widehat{G})$ for some affine group scheme $\widehat{G}$. It is easy to see that the answer to this question is positive: categories of finite-dimensional representations of affine \textit{super}group schemes provide new examples since dimensions of objects in such categories might take negative integer values. In fact, there are STC's with objects of arbitrary complex dimension. Let $t\in\C$. In \cite[Example 1.27]{DM}, Deligne and Milne introduced the category that will be denoted as $\mathcal{D}_{t}$ in this paper. The category $\mathcal{D}_{t}$ may be viewed as an interpolation $"\on{Rep}^{alg}_{\C}(GL_{t})"$ of the category $\on{Rep}^{alg}_{\C}(GL_{n})$ to arbitrary complex $n$. The object in $\mathcal{D}_{t}$ that plays the role of the standard representation $\C^{n}$ in $\on{Rep}^{alg}_{\C}(GL_{n})$ has dimension $t$ (see, e.g., \cite{CoWi} for details). If $t\notin \Z$, then $\mathcal{D}_{t}$ is an STC, otherwise it is not abelian (which is necessary for being an STC in terminology of \cite{EK}), but rather Karoubian category, and we have a tensor functor $G_{n}:\mathcal{D}_{n}\rightarrow \on{Rep}^{alg}_{\C}(GL_{n})$.
In \cite{D1}, Deligne also introduced similar interpolations $\mathcal{D}^{S}_{t}$, $\mathcal{D}^{O}_{t}$  and $\mathcal{D}^{Sp}_{2t}$ of the categories $\on{Rep}_{\C}(S_{n})$ $\on{Rep}^{alg}_{\C}(O_{n})$ and $\on{Rep}^{alg}_{\C}(Sp_{n})$, respectively. The categories $\mathcal{D}_{t}$, $\mathcal{D}^{S}_{t}$, $\mathcal{D}^{O}_{t}$, and $\mathcal{D}^{Sp}_{t}$ are sometimes called the Deligne's categories.

As was shown in \cite{E2}, many Lie-theoretic constructions associated with classical Lie groups $GL_{n}$, $O_{n}$, or $Sp_{2n}$ (e.g. the corresponding Lie algebra $\mathfrak{g}$, symmetric dual pairs, Harish-Chandra modules, the Yangian $Y(\mathfrak{g})$) can be defined using the language of symmetric tensor categories. Then, replacing categories $\on{Rep}^{alg}_{\C}(GL_{n})$, $\on{Rep}^{alg}_{\C}(O_{n})$, and $\on{Rep}^{alg}_{\C}(Sp_{2n})$ with their super analogs, or with $\mathcal{D}_{t}$, $\mathcal{D}^{O}_{t}$,  and $\mathcal{D}^{Sp}_{2t}$, respectively, one naturally obtains the generalization of these constructions to supergroups or more exotic generalizations, which collectively may be called "Lie theory in complex rank". In this paper, we show that the algebra of higher Hamiltonians (the Bethe algebra) of the quantum (rational) Gaudin model associated to the Lie algebra $\gl_{n}$ is one of such constructions, that is, it can be defined in categorical terms and admits an "interpolation" to any complex $n$. 

\subsection{}
The quantum Gaudin model associated to the Lie algebra $\gl_{n}$ was first introduced in \cite{G}, and it is a certain degeneration of the quantum spin chain. Fix $m$ distinct complex numbers $z_{1}\lc z_{m}$. Let $e_{ij}$, $i,j=1\lc n$ be the standard basis of the Lie algebra $\gl_{n}$, that is, $e_{ij}$ is the matrix with $1$ on the intersection of the row $i$ and the column $j$ and $0$ everywhere else. The model is given by the following set of $m$ commuting Hamiltonians:
\[H_{a}=\sum_{\substack{b=1 \\ b\neq a}}^{m}\frac{\sum_{i,j=1}^{n}e^{(a)}_{ij}e^{(b)}_{ji}}{z_{a}-z_{b}}\in U(\gl_{n})^{\otimes m},\quad a=1\lc m,\]
where  $e_{ij}^{(l)}=1\otimes\dots\otimes 1\otimes \underset{l-th}{e_{ij}}\otimes 1\otimes\dots \otimes 1$. 

Let $V_{\nu}$ denote the irreducible representation of $\gl_{n}$ of highest weight $\nu$. Fix a sequence $\bar{\nu}=(\nu^{(1)}\lc\nu^{(m)})$ of integral dominant $\gl_{n}$-weights. The Hamiltonians $H_{a}$, $a=1\lc m$ act on the tensor product $V(\bar{\nu})=V_{\nu^{(1)}}\otimes\dots\otimes V_{\nu^{(m)}}$. Let $V(\bar{\nu})^{sing}$ be the space of singular vectors with respect to the action of $\gl_{n}$ on $V(\bar{\nu})$. The Hamiltonians commute with this action, and therefore, act on $V(\bar{\nu})^{sing}$. The algebra $\mathcal{B}^{sing}_{n}(\bar{\nu})$ of higher Hamiltonians or "conserved charges" of the quantum Gaudin model, which sometimes is also called the Bethe algebra, is the centralizer of $H_{a}$, $a=1\lc m$ in the algebra $\End(V(\bar{\nu})^{sing})$ of all linear transformations of the space $V(\bar{\nu})^{sing}$. 

There is also the universal Bethe algebra $\mathcal{B}_{n}\in U(\gl_{n})^{\otimes m}$, see \cite{FFR}, \cite{Tal}, \cite{MTV6}. The algebra $\mathcal{B}_{n}$ commutes with the copy of $U(\gl_{n})$ in $U(\gl_{n})^{\otimes m}$ given by the inclusion $g\mapsto\sum_{a=1}^{m}1\otimes\dots\otimes 1\otimes \underset{a-th}{g}\otimes 1\otimes\dots \otimes 1$, and therefore, acts on $V(\bar{\nu})^{sing}$. The image of $\mathcal{B}_{n}$ in $\End(V(\bar{\nu})^{sing})$ under this action coincides with $\mathcal{B}^{sing}_{n}(\bar{\nu})$. The algebra $\mathcal{B}_{n}$ can be constructed using the center of the local completion of the universal enveloping algebra of the affine Lie algebra $\widehat{\gl}_{n}$, see \cite{FFR}. In this paper, we will use more explicit construction of $\mathcal{B}_{n}$ given in \cite{Tal} and \cite{CF}, see Section \ref{3} for details. Also, instead of $\mathcal{B}^{sing}_{n}(\bar{\nu})$, we will consider the image $\mathcal{B}_{n}(\bar{\nu})$ of $\mathcal{B}_{n}$ in $\End(V(\bar{\nu}))$ since we have an isomorphism of algebras 
$\mathcal{B}^{sing}_{n}(\bar{\nu})\cong \mathcal{B}_{n}(\bar{\nu})$.

Now, let us return to the Deligne's category $\mathcal{D}_{t}$. Indecomposable objects of $\mathcal{D}_{t}$ are labeled by pairs of partitions. For each integral dominant $\gl_{n}$-weight $\nu$, there is a pair of partitions $(\lambda, \mu)$ such that $G_{n}(V_{(\lambda,\mu)})=V_{\nu}$, where $V_{(\lambda,\mu)}$ is the indecomposable object corresponding to $(\lambda,\mu)$, and $G_{n}:\mathcal{D}_{n}\rightarrow\on{Rep}^{alg}_{\C}(GL_{n})$ is the functor mentioned above, see details in Section \ref{2}.
For a sequence $\bar{\nu}=(\nu^{(1)}\lc\nu^{(m)})$ of integral dominant $\gl_{n}$-weights, let $(\lambda^{(a)},\mu^{(a)})$, $a=1\lc m$ be such that $G_{n}(V_{(\lambda^{(a)},\mu^{(a)})})=V_{\nu^{(a)}}$. Denote $V(\bar{\lambda},\bar{\mu})=V_{(\lambda^{(1)},\mu^{(1)})}\otimes\dots\otimes V_{(\lambda^{(m)},\mu^{(m)})}$. In Section \ref{4}, we construct an algebra $\underline{\mathcal{B}}_{t}(\bar{\lambda},\bar{\mu})\in\End_{\mathcal{D}_{t}}(V(\bar{\lambda},\bar{\mu}))$ such that the functor $G_{n}$ induces a homomorphism $\underline{\mathcal{B}}_{n}(\bar{\lambda},\bar{\mu})\rightarrow\mathcal{B}_{n}(\bar{\nu})$, which becomes an isomorphism for sufficiently large $n$. 

\subsection{}
A standard problem in the theory of quantum integrable models is the problem of simultaneous diagonalization of higher Hamiltonians. One way to find common eigenvectors and eigenvalues of the elements of $\mathcal{B}_{n}(\bar{\nu})$ is to use the algebraic Bethe ansatz, see \cite{MTV6}. Each eigenvector obtained in this way is associated with a solution of a system of algebraic equations called the Bethe ansatz equations. The Bethe ansatz conjecture states that all common eigenvectors can be obtained in this way. 

In \cite{FFR}, the Bethe ansatz equations for the Gaudin model associated with a semisimple Lie algebra $\mathfrak{g}$ were interpreted as "no-monodromy" conditions on a certain space of ${}^{L}\mathfrak{g}$-opers $\on{Op}(z_{1}\lc z_{m})$, where ${}^{L}\mathfrak{g}$ is the Langlands dual of  $\mathfrak{g}$, see also \cite{FrG}. The Bethe ansatz conjecture was reformulated in \cite{FFR} as follows: there is a bijection between the set of common eigenvectors of the higher Hamiltonians and the set of monodromy-free ${}^{L}\mathfrak{g}$-opers from $\on{Op}(z_{1}\lc z_{m})$. This conjecture was proved for $\mathfrak{g}=\gl_{n}$ in \cite{MTV8}, and for any semisimple Lie algebra $\mathfrak{g}$ in \cite{Ryb}.

In the case $\mathfrak{g}=\gl_{n}$, we have ${}^{L}\mathfrak{g}=\mathfrak{g}$, and a $\gl_{n}$-oper from $\on{Op}(z_{1}\lc z_{m})$ may be viewed as a differential operator of the form
\[ D=\partial_{u}^{n}+\sum_{i=1}^{n}\sum_{j=1}^{i}\sum_{a=1}^{m}\frac{b_{ij}^{(a)}}{(u-z_{a})^{j}}\partial_{u}^{n-i}, \]
where we denote $\partial_{u}=d/du$. Then the no-monodromy conditions discussed above are polynomial relations  on the coefficients $b_{ij}^{(a)}$ saying that for each $a=1\lc m$, $i=1\lc n$, the differential operator $D$ has formal series solutions of the form 
\[ f_{a,i}(u)=(u-z_{a})^{n+\nu_{i}^{(a)}-i}+c(u-z_{a})^{n+\nu_{i}^{(a)}-i+1}+\dots,\]
and that it has regular singularity at $\infty$, see details in Section \ref{5}. Here, $\nu^{(a)}=(\nu_{1}^{(a)}\lc\nu_{n}^{(a)})$ is the $a$-th $\gl_{n}$-weight in the sequence $\bar{\nu}$ used in the construction of $\mathcal{B}_{n}(\bar{\nu})$. 

Let $p_{1}\lc p_{d}$ be the polynomials such that the no-monodromy conditions are written as $p_{k}(b_{ij}^{(a)})=0$, $k=1\lc d$. Consider the free commutative unital algebra $\mathcal{A}_{n}$ with free generators $y_{ij}^{(a)}$, $i=1\lc n$, $j=1\lc i$, $a=1\lc m$. Let $J_{n}(\bar{\nu})$ be the ideal in $\mathcal{A}_{n}$ generated by the elements $p_{k}(y_{ij}^{(a)})$, $k=1\lc d$. The Bethe ansatz conjecture was proved in \cite{Ryb} by proving a theorem, which in case of $\mathfrak{g}=\gl_{n}$ may be formulated as follows.
\begin{thm}\label{intro thm}
The algebras $B_{n}(\bar{\nu})$ and $\mathcal{A}_{n}/J_{n}(\bar{\nu})$ are isomorphic.
\end{thm}

In the second part of the paper, we formulate and prove an analog of this theorem for the algebra $\underline{\mathcal{B}}_{t}(\bar{\lambda},\bar{\mu})$ with some restrictions on the sequence of bipartitions $(\bar{\lambda},\bar{\mu})$, see formula \eqref{condition on bipartitions}. In this case, the order of the differential operator $D$ becomes an arbitrary complex number, so $D$ becomes a pseudo-differential operator, i. e., a formal series of the form 
\[ \partial_{u}^{t}+\sum_{i=1}^{\infty}\sum_{j=1}^{i}\sum_{a=1}^{m}\frac{b_{ij}^{(a)}}{(u-z_{a})^{j}}\partial_{u}^{t-i}. \]
Therefore, we consider an infinitely generated free commutative unital algebra $\mathcal{A}_{\infty}$ with free generators $y_{ij}^{(a)}$, $i\in\Z_{\geq 1}$, $j=1\lc i$, $a=1\lc m$, and construct an ideal $\underline{J}_{t}(\bar{\lambda},\bar{\mu})$ in $\mathcal{A}_{\infty}$ such that for sufficiently large $n$, the projection of $\underline{J}_{t=n}(\bar{\lambda},\bar{\mu})$ to $\mathcal{A}_{n}$ coincides with $J_{n}(\bar{\nu})$, where $\bar{\nu}$ is such that $G_{n}(V_{(\lambda^{(a)},\mu^{(a)})})=V_{\nu^{(a)}}$. Then we prove that for all but finitely many $t$, the algebras $\underline{\mathcal{B}}_{t}(\bar{\lambda},\bar{\mu})$ and $\mathcal{A}_{\infty}/\underline{J}_{t}(\bar{\lambda},\bar{\mu})$ are isomorphic (Theorem \ref{main}). The key observation is that the algebra $\mathcal{A}_{\infty}/\underline{J}_{t}(\bar{\lambda},\bar{\mu})$ is finitely generated and for sufficiently large integers $n$, the ideal $\underline{J}_{n}(\bar{\lambda},\bar{\mu})$ contains all elements $y_{ij}^{(a)}$ with $i>n$, see Lemma \ref{tool}.

At the moment, it is unclear to us how to construct the ideal $\underline{J}_{t}(\bar{\lambda},\bar{\mu})$ for any $(\bar{\lambda},\bar{\mu})$ (not necessarily satisfying condition \eqref{condition on bipartitions}). We indicate some technical challenges that appear in this case in Remarks \ref{R1} and \ref{R2}.

\subsection{}
Our study of the algebra $\underline{\mathcal{B}}_{t}(\bar{\lambda},\bar{\mu})$ is in particular motivated by the Bethe ansatz conjecture for the Bethe algebra $\mathcal{B}_{n|n'}$ associated with the super Lie algebra $\gl_{n|n'}$, see \cite{MVY}, \cite{HMVY}. This conjecture relates eigenvectors of $\mathcal{B}_{n|n'}$ with pseudo-differential operators of order $n-n'$, which are ratios of monodromy-free differential operators. These ratios are described by some polynomial relations on coefficients of pseudo-differential operators, and we expect that these relations are exactly the relations in the algebra $\mathcal{B}_{n|n'}$.

On the other hand, there is a full (but not faithful) tensor functor $G_{n|n'}:\mathcal{D}_{n-n'}\rightarrow \on{Rep}_{\C}(GL_{n|n'})$ to the category of finite-dimensional representations of the supergroup $GL_{n|n'}$ such that the functor $G_{n}$ discussed above is the special case $n'=0$ of $G_{n|n'}$. The existence of the functor $G_{n|n'}$ is a consequence of a certain universal property of the Deligne's category $\mathcal{D}_{t}$ (see \cite[Proposition 3.5.1]{CoWi}). In fact, it was shown in \cite{AHS} that for $t\in\Z$, the abelian envelope of $\mathcal{D}_{t}$ is a certain limit of the categories $\on{Rep}_{\C}(GL_{n|n'})$ as $n,n'\rightarrow \infty$, $n-n'=t$. It is easy to see that the functor $G_{n|n'}$ induces a surjective homomorphism of algebras $\phi_{n\vert n'}:\,\underline{\mathcal{B}}_{n-n'}(\bar{\lambda},\bar{\mu})\rightarrow \mathcal{B}_{n|n'}$. Therefore, the ideal $\underline{J}_{n-n'}(\bar{\lambda},\bar{\mu})$ together with the kernel of the homomorphism $\phi_{n\vert n'}$ give all relations in the algebra $\mathcal{B}_{n|n'}$. In Section \ref{7}, we prove that the relations given by the ideal $\underline{J}_{n-n'}(\bar{\lambda},\bar{\mu})$ are satisfied by the ratios of monodromy-free differential operators arising in \cite{HMVY}. It follows that for proving the Bethe ansatz conjecture for the Gaudin model in super case, it is important to describe the kernel of $\phi_{n\vert n'}$. We would like to do this in our future studies.  

\subsection{}
The paper is organized as follows. In Section \ref{2}, we give the definition of the Deligne's category $\mathcal{D}_{t}$ and discuss its properties that will be used later. In Section \ref{3}, we recall the construction of the Bethe algebra $\mathcal{B}_{n}$, and present different generating sets of this algebra. In Section \ref{4}, we introduce the Bethe algebra $\underline{\mathcal{B}}_{t}(\bar{\lambda},\bar{\mu})$ in the Deligne's category $\mathcal{D}_{t}$ as well as its generating sets analogous to the generating sets of $\mathcal{B}_{n}$ introduced in Section \ref{3}. Section \ref{5} is devoted to the "no-monodromy" conditions, that is, to the construction of the ideal $J_{n}(\bar{\nu})$ and its relation to the Bethe algebra $\mathcal{B}_{n}(\bar{\nu})$. We construct the ideal $\underline{J}_{t}(\bar{\lambda},\bar{\mu})$ and prove an analog of Theorem \ref{intro thm} for $\underline{\mathcal{B}}_{t}(\bar{\lambda},\bar{\mu})$ in Section \ref{6}. We show how our interpolations of the no-monodromy conditions is related to ratios of differential operators in Section \ref{7}.

\subsection{Acknowledgments}
This article is an output of a research project implemented as part of the Basic Research Program at the National Research University Higher School of Economics (HSE University).

\section{Preliminaries on the Deligne's category $\mathcal{D}_{t}$}\label{2}
\subsection{}
Let $f$ and $f'$ be finite (possibly empty) words in two letters: $\bullet$ and $\circ$. Following \cite{CoWi}, define an $(f,f')$-diagram to be a graph satisfying the next five properties:
\begin{enumerate}
    \item Each vertex is either $\bullet$ or $\circ$.
    \item The vertices are positioned in two rows so that in the top row, we obtain the word $f$, and in the bottom row, we obtain the word $f'$. 
    \item Each vertex is adjacent to exactly one edge.
    \item If an edge connects two vertices in the same row, then it connects a black vertex $\bullet$ and a white vertex $\circ$.
    \item If an edge connects two vertices in different rows, then it connects either two black or two white vertices. 
\end{enumerate}

\begin{example}
If $f=\bullet\quad\bullet\quad\circ\quad\bullet\quad\circ\quad\circ$ and $f'=\bullet\quad\bullet\quad\circ\quad\circ$, then 
\vspace{0.5cm}
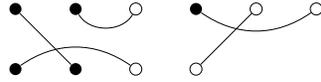
\begin{center}
   \scalebox{0.8} {\begin{tikzpicture}[
V/.style={fill,circle,inner sep=0pt,minimum size=2mm},
V*/.style={draw, circle,inner sep=0pt,minimum size=2mm},
]

\node[V] (1) at (0,0) {};
\node[V] (2) at (1,0) {};
\node[V*] (3) at (2,0) {};
\node[V] (4) at (3,0) {};
\node[V*] (5) at (4,0) {};
\node[V*] (6) at (5,0) {};

\node[V] (1a) at (0,-1) {};
\node[V] (2a) at (1,-1) {};
\node[V*] (3a) at (2,-1) {};
\node[V*] (4a) at (3,-1) {};

\draw[-] (1) to (2a);
\draw[-] (5) to (4a);

\draw[bend right=60,-] (2) to (3);
\draw[bend right=35,-] (4) to (6);
\draw[bend left=35,-] (1a) to (3a);

 \end{tikzpicture}}
\end{center} 
is an $(f,f')$-diagram.
\end{example}

Fix $t\in\C$. Define a $\C$-linear strict monoidal category $\mathcal{D}_{t}^{0}$ as follows.
\begin{itemize}
    \item Objects of $\mathcal{D}_{t}^{0}$ are finite (possibly empty) words in letters $\bullet$ and $\circ$.
    \item If $f$ and $f'$ are two objects in $\mathcal{D}_{t}^{0}$, then $\Hom_{\mathcal{D}_{t}^{0}}(f,f')$ is a vector space with a basis consisting of all $(f,f')$-diagrams. 
    \item If $X$ is an $(f,f')$-diagram and  $Y$ is an $(f',f'')$-diagram, define the composition $Y\circ X$ as follows:
    \begin{enumerate}
        \item stack $X$ atop $Y$ so that the bottom row of $X$ is identified with the top row of $Y$,
        \item erase any loop that appeared in step (1),
        \item multiply the obtained $(f,f'')$-diagram by $t^{d}$, where $d$ is the number of loops erased in step (2).
    \end{enumerate}
    We extend the composition to any morphisms $X\in\Hom_{\mathcal{D}_{t}^{0}}(f,f')$ and $Y\in\Hom_{\mathcal{D}_{t}^{0}}(f',f'')$ by linearity.
    \item Tensor product is given by superposition of words or diagrams from left to right.
\end{itemize}
\begin{example}
If
\[X=
\vcenter{\hbox{
\scalebox{0.8} {\begin{tikzpicture}[
V/.style={fill,circle,inner sep=0pt,minimum size=2mm},
V*/.style={draw, circle,inner sep=0pt,minimum size=2mm},
]

\node[V] (1) at (0,0) {};
\node[V] (2) at (1,0) {};
\node[V*] (3) at (2,0) {};

\node[V] (1a) at (0,-1) {};
\node[V*] (2a) at (1,-1) {};
\node[V] (3a) at (2,-1) {};

\draw[-] (1) to (3a);

\draw[bend right=60,-] (2) to (3);
\draw[bend left=60,-] (1a) to (2a);

 \end{tikzpicture}}}}
 \quad\text{and}\quad 
 Y=
\vcenter{\hbox{
\scalebox{0.8} {\begin{tikzpicture}[
V/.style={fill,circle,inner sep=0pt,minimum size=2mm},
V*/.style={draw, circle,inner sep=0pt,minimum size=2mm},
]

\node[V] (1) at (0,0) {};
\node[V*] (2) at (1,0) {};
\node[V] (3) at (2,0) {};

\node[V] (1a) at (0,-1) {};
\node[V*] (2a) at (1,-1) {};
\node[V] (3a) at (2,-1) {};

\draw[-] (3) to (3a);

\draw[bend right=60,-] (1) to (2);
\draw[bend left=60,-] (1a) to (2a);

 \end{tikzpicture}}}}\,\, ,
 \]
 then
 \[Y\circ X = \vcenter{\hbox{
\scalebox{0.8} {\begin{tikzpicture}[
V/.style={fill,circle,inner sep=0pt,minimum size=2mm},
V*/.style={draw, circle,inner sep=0pt,minimum size=2mm},
]

\node[V] (1) at (0,0) {};
\node[V] (2) at (1,0) {};
\node[V*] (3) at (2,0) {};

\node[V] (1a) at (0,-1) {};
\node[V*] (2a) at (1,-1) {};
\node[V] (3a) at (2,-1) {};

\node[V] (1b) at (0,-2) {};
\node[V*] (2b) at (1,-2) {};
\node[V] (3b) at (2,-2) {};

\draw[-] (1) to (3a);
\draw[-] (3a) to (3b);

\draw[bend right=60,-] (2) to (3);
\draw[bend left=60,-] (1a) to (2a);
\draw[bend right=60,-] (1a) to (2a);
\draw[bend left=60,-] (1b) to (2b);

 \end{tikzpicture}}}}\quad =\quad 
 t\cdot
 \vcenter{\hbox{
\scalebox{0.8} {\begin{tikzpicture}[
V/.style={fill,circle,inner sep=0pt,minimum size=2mm},
V*/.style={draw, circle,inner sep=0pt,minimum size=2mm},
]

\node[V] (1) at (0,0) {};
\node[V] (2) at (1,0) {};
\node[V*] (3) at (2,0) {};

\node[V] (1a) at (0,-1) {};
\node[V*] (2a) at (1,-1) {};
\node[V] (3a) at (2,-1) {};

\draw[-] (1) to (3a);

\draw[bend right=60,-] (2) to (3);
\draw[bend left=60,-] (1a) to (2a);

 \end{tikzpicture}}}}\,\, .
 \]
\end{example}

We will also write $V$ (resp., $V^{*}$) for the object $\bullet$ (resp., $\circ$) of $\mathcal{D}_{t}^{0}$. Notice that as a monoidal category, $\mathcal{D}_{t}^{0}$ is generated by $V$ and $V^{*}$. Denote $T^{r,s}\coloneqq V^{\otimes r}\otimes(V^{*})^{\otimes{s}}$ and $\mathfrak{g} \coloneqq T^{1,1} = V\otimes V^{*}$.
\subsection{}
Consider the additive envelope $(\mathcal{D}_{t}^{0})^{add}$ of the category $\mathcal{D}_{t}^{0}$. By definition, $(\mathcal{D}_{t}^{0})^{add}$ is a category with objects $X=(X_{j})$ given by finite-length tuples of objects from $\mathcal{D}_{t}^{0}$ and morphisms
\[ \phi: (X_{j})\rightarrow (Y_{i}),\quad \phi=(\phi_{ij}),\quad \phi_{ij}\in\Hom_{\mathcal{D}_{t}^{0}} (X_{j}, Y_{i})\]
given by matrices of morphisms from $\mathcal{D}_{t}^{0}$ and composed via the matrix multiplication:
\[(\phi\circ\psi)_{ij}=\sum_{k}\phi_{ik}\circ\psi_{kj}.\]

The category $(\mathcal{D}_{t}^{0})^{add}$ is additive with biproducts given by concatenation of tuples. It also naturally inherits $\C$-linear and monoidal structures from the category $\mathcal{D}_{t}^{0}$.

By definition, the Deligne's category $\mathcal{D}_{t}$ is the Karoubi envelope of $\mathcal{D}_{t}$. That is, objects of $\mathcal{D}_{t}$ are pairs $(X, e)$, where $X$ is an object from $\mathcal{D}_{t}^{0}$, $e$ is an idempotent in $\End_{(\mathcal{D}_{t}^{0})^{add}}(X)$, and 
\[\Hom ((X,e),(Y,f))=\{\phi\in\Hom_{(\mathcal{D}_{t}^{0})^{add}}(X,Y)\vert f\circ\phi= \phi = \phi\circ e\}.\]

The composition of morphisms in $\mathcal{D}_{t}$ is inherited from $(\mathcal{D}_{t}^{0})^{add}$. Also, the category $\mathcal{D}_{t}$ naturally inherits monoidal and additive structures from $(\mathcal{D}_{t}^{0})^{add}$.

\subsection{}\label{prelim Deligne 3}
A partition $\lambda=(\lambda_{1},\lambda_{2},\dots)$ is an infinite non-increasing sequence of non-negative integers stabilizing at zero. Denote $|\lambda|=\sum_{i=1}^{\infty}\lambda_{i}$. We will also write $\emptyset$ for the partition $(0,0,\dots)$. A bipartition is a pair of partitions. 

It was proved in \cite{CoWi} that any indecomposable object of the Deligne's category $\mathcal{D}_{t}$ is isomorphic to $V_{(\lambda,\mu)}=(T^{r,s}, e_{(\lambda,\mu)})$ for some bipartition $(\lambda,\mu)$, where  $r=|\lambda |$, $s=|\mu |$, and $e_{(\lambda,\mu)}$ is a certain idempotent in $\End _{\mathcal{D}_{t}^{0}}(T^{r,s})$, see \cite[Section 4.3]{CoWi} for details. 

Notice that for any $m\in\Z_{> 0}$, the algebras $\End_{\mathcal{D}_{t}}(V^{\otimes m})$ and $\End_{\mathcal{D}_{t}}((V^{*})^{\otimes m})$ are isomorphic to the group algebra $\C [S_{m}]$ of the permutation group $S_{m}$. Under this isomorphism, the idempotents $e_{(\lambda,\emptyset)}$ and $e_{(\emptyset,\lambda)}$ with $|\lambda |=m$ correspond to a scalar multiple of the Young symmetrizer $c_{\lambda}$ (see, e.g., \cite[Lecture 4]{FH} for the definition of $c_{\lambda}$). In particular, $e_{(\lambda,\emptyset)}$ and $e_{(\emptyset,\lambda)}$ do not depend on the complex number $t$. 

For each $n\in\Z_{>0}$, let $V_{\sq}$ be the tautological representation of the Lie algebra $\gl_{n}$. Recall that a rational representation of $\gl_{n}$ is a subrepresentation of $V_{\sq}^{\otimes r} \otimes (V_{\sq}^{*})^{\otimes s}$ for some $r,s\in\Z_{\geq 0}$. The Deligne's category $\mathcal{D}_{t}$ is an "interpolation" of the category $\mathcal{R}_{n}$ of rational representations of $\gl_{n}$ in the sense that there exists an additive modular full essentially surjective functor $G_{n}:\mathcal{D}_{n}\rightarrow \mathcal{R}_{n}$, see \cite[Section 5]{CoWi} for details. 

For a bipartition $(\lambda, \mu)$, let $l_{\lambda}$ and $l_{\mu}$ be the numbers of nonzero elements in partitions $\lambda$ and $\mu$, respectively. If $n\geq l_{\lambda}+l_{\mu}$, then the functor $G_{n}$ sends the indecomposable object $V_{(\lambda,\mu)}$ to the irreducible $\gl_{n}$-module $V_{\nu}$ of highest weight $\nu$ defined by $\lambda=(\lambda_{1},\lambda_{2},\dots)$ and $\mu=(\mu_{1},\mu_{2},\dots)$ as follows:
\begin{equation}\label{lambda mu nu}
    \nu=(\lambda_{1},\lambda_{2}\lc\lambda_{l_{\lambda}},0,0\lc 0,-\mu_{l_{\mu}},\lc-\mu_{2}, -\mu_{1}).
\end{equation}
If $n<l_{\lambda}+l_{\mu}$, then $G_{n}(V_{(\lambda,\mu)})=\{0\}$. 

Fix $m$ bipartitions $(\lambda^{(a)},\mu^{(a)})$, $a=1\lc m$. 
Denote $\bar{V}=V_{(\lambda^{(1)},\mu^{(1)})}\otimes V_{(\lambda^{(2)},\mu^{(2)})}\otimes\dots \otimes V_{(\lambda^{(m)},\mu^{(m)})}$ and $N=\sum_{a=1}^{m}(|\lambda^{(a)} |+|\mu^{(a)} |)$. Let $G_{n}^{\bar{V}}:\End_{\mathcal{D}_{n}}(\bar{V})\rightarrow \End_{\mathcal{R}_{n}}(G_{n}(\bar{V}))$ be the surjective homomorphism induced by the functor $G_{n}$. 

\begin{lem}\label{G_n inj}
For each $n\geq N$, $G_{n}^{\bar{V}}$ is an isomorphism.
\end{lem}
\begin{proof}
Notice that $G_{n}^{\bar{V}}$ is a restriction of $G_{n}^{\bar{W}}$, where
\[
\bar{W}=V^{\otimes |\lambda^{(1)} |}\otimes (V^{*})^{\otimes |\mu^{(1)}|}\otimes \dots  \otimes V^{\otimes |\lambda^{(m)} |}\otimes (V^{*})^{\otimes |\mu^{(m)}|}.
\]
Therefore, it is enough to show that $G_{n}^{\bar{W}}$ is injective.

For any objects $A$, $B$, and $C$ of $\mathcal{D}_{n}$ or $\mathcal{R}_{n}$, we have $A\otimes B\cong B\otimes A$ and $\Hom (A, B\otimes C) \cong \Hom (A\otimes C^{*}, B)$. Thus, we have vector space isomorphisms $h_{1}:\, \End_{\mathcal{D}_{n}}(\bar{W})\rightarrow\End_{\mathcal{D}_{n}}(V^{\otimes N})$ and $h_{2}:\, \End_{\mathcal{R}_{n}}(G_{n}(\bar{W}))\rightarrow\End_{\mathcal{R}_{n}}(V_{\sq}^{\otimes N})=\C[S_{N}]$. The last equality follows from the Schur-Weyl duality and our assumption $n\geq N$. The map $h_{2}G_{n}^{\bar{W}}h_{1}^{-1}$ sends different diagrams to different elements of $S_{N}$, therefore, it is injective, and the lemma follows. 
\end{proof}

\section{The Bethe algebra of the Gaudin model associated with $\gl_{n}$}\label{3}
\subsection{}\label{Bethe 1}
Let $e_{ij}$, $i,j=1\lc n$ be the standard basis of the Lie algebra $\gl_{n}$, that is, $e_{ij}$ is the matrix with $1$ on the intersection of the row $i$ and the column $j$ and $0$ everywhere else. Fix distinct complex numbers $z_{1}\lc z_{m}$. For each $i,j=1\lc n$, define a $U(\gl_{n})^{\otimes m}$-valued rational function $\mathcal{L}_{ij}(u)$ of a variable $u$ by the formula:
\[\mathcal{L}_{ij}(u)=\sum_{r=1}^{m}\frac{e_{ij}^{(r)}}{u-z_{r}},\]
where  $e_{ij}^{(r)}=1\otimes\dots\otimes 1\otimes \underset{r-th}{e_{ij}}\otimes 1\otimes\dots \otimes 1$.

Let $\mathcal{L}(u)$ be an $n\times n$ matrix with entries $\mathcal{L}_{ij}(u)$. Consider a differential operator
\[
\tr \left( \partial_{u}-\mathcal{L}(u)\right)^{k}=\sum_{l=0}^{k}S_{kl}(u)\partial_{u}^{k-l}.
\]
We have $S_{k0}(u)=1$, and
\begin{equation}\label{S_kl}
S_{kl}(u)=\sum_{r=1}^{m}\sum_{j=1}^{l}\frac{S_{klj}^{(r)}}{(u-z_{r})^{j}},\quad 1\leq l\leq k
\end{equation}
for some $S_{klj}^{(r)}\in U(\gl_{n})^{\otimes m}$. It is known, see for example, \cite{CF}, that the elements $S_{klj}^{(r)}$, $k\geq 1$, $1\leq l\leq k$, $1\leq j\leq l$, $1\leq r\leq m$ commute with each other.

Let $\mathcal{B}_{n}$ be the unital commutative subalgebra of  $U(\gl_{n})^{\otimes m}$ generated by the elements $S_{klj}^{(r)}$, $k\geq 1$, $1\leq l\leq k$, $1\leq j\leq l$, $1\leq r\leq m$. The algebra $\mathcal{B}_{n}$ is known as the algebra of higher Hamiltonians (the Bethe algebra) of the quantum rational Gaudin model, see \cite{FFR}, \cite{Tal}, \cite{CF}, \cite{MTV6}.

Recall that $V_{\nu}$ denotes the finite dimensional representation of the Lie algebra $\gl_{n}$ with highest weight $\nu$. Let $\bar{\nu}=(\nu^{(1)}\lc\nu^{(m)})$ be a sequence of dominant integral $\gl_{n}$-weights. The Bethe algebra $\mathcal{B}_{n}$ acts on the tensor product $V(\bar{\nu})=V_{\nu^{(1)}}\otimes\dots \otimes V_{\nu^{(m)}}$. Denote by $\mathcal{B}_{n}(\bar{\nu})$ the projection of $\mathcal{B}_{n}$ in $\End(V(\bar{\nu}))$. In Section \ref{4} below, we will construct an analog of the algebra $\mathcal{B}_{n}(\bar{\nu})$ in the Deligne's category $\mathcal{D}_{t}$.

\subsection{}\label{Bethe 2}
In what follows, we will need another generating set for the algebra $\mathcal{B}_{n}$, which we are going to present in this subsection.

For an $n\times n$ matrix $A$ with possibly non-commuting entries $a_{ij}$, $i,j=1\lc n$, define its column determinant $\cdet A$ by the formula
\[
\cdet A = \sum_{\sigma\in S_{n}}\sign(\sigma)\,a_{\sigma (1) 1}a_{\sigma (2) 2}\dots a_{\sigma (n) n}.
\]

The matrix $A$ is called a Manin matrix if $[a_{ij},a_{pq}]=[a_{pj},a_{iq}]$ for all $i,j,p,q=1\lc n$. 
It is known, see \cite{CF}, that $\partial_{u} - \mathcal{L}(u)$ is a Manin matrix. 

We will need the following lemma:
\begin{lem}\label{Newton}
Suppose that $A$ is a Manin matrix. Denote $T_{k}=\tr A^{k}$. Let $\sigma_{1}\lc\sigma_{n}$ be the coefficients of the "characteristic polynomial":
\[ \cdet (1+\alpha A) = \sum_{k=0}^{n}\sigma_{k}\alpha^{k}.\]
Then the Newton identities 
\[ 
\sigma_{k}=(-1)^{k+1}\frac{1}{k}\sum_{i=0}^{k-1}(-1)^{i}\sigma_{i}T_{k-i}
\]
hold for every $k=1\lc n$.
\end{lem}
The proof of the lemma can be found in \cite{CF}.

Consider a differential operator
\begin{equation}\label{D(w)}
\cdet \bigl[ 1 + \alpha\bigl(\partial_{u}-\mathcal{L}(u)\bigr)\bigr]=\sum_{r=0}^{n}\sum_{s=0}^{r}\alpha^{r}\tilde{B}_{rs}(u)\partial_{u}^{r-s}.
\end{equation}
The coefficients $\tilde{B}_{rs}(u)$ are $U(\gl_{n})^{\otimes m}$-valued rational functions of $u$ with possible poles only at $z_{1}\lc z_{m}$. For all $r=0\lc n$, we have $\tilde{B}_{r0}(u)=1$ and 
\begin{equation}\label{B(u)}
\tilde{B}_{rs}(u)=\sum_{a=1}^{m}\sum_{j=1}^{s}\frac{\tilde{B}_{rsj}^{(a)}}{(u-z_{a})^{j}},\quad 1\leq s\leq r 
\end{equation}
for some $\tilde{B}_{rsj}^{(a)}\in U(\gl_{n})^{\otimes m}$.

Applying Lemma \ref{Newton} to the matrix $\partial_{u}-\mathcal{L}(u)$, we see that the elements $\tilde{B}_{rsj}^{(a)}$, $r=1\lc n$, $s=1\lc r$, $j=1\lc s$, $a=1\lc m$ generate the Bethe algebra $\mathcal{B}_{n}$.

Denote $B_{i}(u)\coloneqq\tilde{B}_{ii}(u)$ and $B^{(a)}_{ij}\coloneqq\tilde{B}_{iij}^{(a)}$.
\begin{prop}\label{B}
The elements $B_{ij}^{(a)}$, $i=1\lc n$, $j=1\lc i$, $a=1\lc m$ generate the algebra $\mathcal{B}_{n}$.
\end{prop}
\begin{proof}
Let $\mathcal{D}$ be the algebra of differential operators with respect to the variable $u$ with $U(\gl_{n})^{\otimes m}$-valued rational coefficients. Introduce a map $\Phi$: $\mathcal{D}\rightarrow\mathcal{D}[\alpha^{-1}]$
\[ \Phi:\,\sum_{r}a_{r}(u)\partial_{u}^{r}\mapsto\sum_{r}a_{r}(u)(\partial_{u}+\alpha^{-1})^{r}.\]

It is straightforward to check that $\Phi$ is a homomorphism of algebras. Therefore,
\begin{equation}\label{Phi cdet}
    \begin{split}
    \Phi\bigl[\cdet\bigl(\partial_{u}-\mathcal{L}(u)\bigr)\bigr] & =\cdet\bigl[\Phi\bigl(\partial_{u}-\mathcal{L}(u)\bigr)\bigr] = \\
    & = \alpha^{-n}\cdet\bigl[ 1+\alpha\bigl(\partial_{u}-\mathcal{L}(u)\bigr)\bigr].
    \end{split}
\end{equation}

Using the relation 
\[(\partial_{u}+\alpha^{-1})^{r}=\sum_{s=0}^{r}\binom{r}{s}\alpha^{-r+s}\partial_{u}^{s}, \]
one can check that $\sum_{r=0}^{N}\sum_{s=0}^{M}a_{rs}\alpha^{-N+r}\partial_{u}^{s}\in\mathcal{D}[\alpha^{-1}]$ belongs to the image of $\Phi$ if and only if $a_{rs}=0$ for $s>r$, and 
\[a_{rs}=a_{r-s,0}\binom{N-r+s}{s}.\]

Since by \eqref{Phi cdet}, the differential operator $\alpha^{-n}\cdet\bigl[ 1+\alpha\bigl(\partial_{u}-\mathcal{L}(u)\bigr)\bigr]$ belongs to the image of $\Phi$, we have 
\begin{equation}\label{B relation}
\tilde{B}_{rs}(u)=\tilde{B}_{ss}(u)\binom{n-s}{r-s}
\end{equation}
for all $r=0\lc n$, $s=0\lc r$, and the proposition follows.
\end{proof}

\begin{rem}
Notice that the generators $B_{ij}^{(a)}$ of $\mathcal{B}_{n}$ are the coefficients in the following expansion:
\begin{equation}\label{cdet expansion}
\cdet \bigl( \partial_{u} - \mathcal{L}(u)\bigr) = \sum_{i=1}^{n}\sum_{j=1}^{i}\sum_{a=1}^{m}\frac{B_{ij}^{(a)}}{(u-z_{a})^{j}}\partial_{u}^{n-i}.
\end{equation}

Indeed, we have 
\[\cdet \bigl( \partial_{u} - \mathcal{L}(u)\bigr) = \sum_{i=1}^{n}B_{ni}(u)\partial_{u}^{n-i}.\]
Taking $n=s$ in \eqref{B relation}, we get $B_{ns}(u)=B_{ss}(u)$. Since $B_{ii}(u)=\sum_{j=1}^{i}\sum_{a=1}^{m}\frac{B_{ij}^{(a)}}{(u-z_{a})^{j}}$, we obtain \eqref{cdet expansion}
\end{rem}

\section{The Bethe algebra in the Deligne's category $\mathcal{D}_{t}$}\label{4}
\subsection{}
Recall the object  $\g=V\otimes V^{*}$ of the Deligne's category $\mathfrak{D}_{t}$. 

Let $\mathcal{C}_{m}\in\Hom_{\mathcal{D}_{t}}{(\C, \g ^{\otimes m})}$ be the morphism given by the diagram

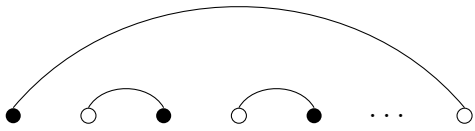
\begin{center}
\begin{tikzpicture}[
V/.style={fill,circle,inner sep=0pt,minimum size=2mm},
V*/.style={draw, circle,inner sep=0pt,minimum size=2mm},
]

\node[V] (1) at (0,0) {};
\node[V*] (2) at (1,0) {};
\node[V] (3) at (2,0) {};
\node[V*] (4) at (3,0) {};
\node[V] (5) at (4,0) {};
\node[] (dots) at (5,0) {$\dots$};
\node[V*] (6) at (6,0) {};

\draw[bend left=60,-] (2.north) to (3.north);
\draw[bend left=60,-] (4.north) to (5.north);
\draw[bend left=50,-] (1.north) to (6.north);

\end{tikzpicture}
\end{center}

The morphism $\mathcal{C}_{m}$ plays the role of the element
\begin{equation*}
C_{m}=\sum_{i_{1}, i_{2}\lc i_{m}} e_{i_{1}i_{2}}^{(1)} e_{i_{2}i_{3}}^{(2)}\dots e_{i_{m}i_{1}}^{(m)}\in \gl_{n}^{\otimes m}.
\end{equation*}

\vspace{0.5cm}

 Fix $k\in\Z_{>0}$, and a sequence $\bar{i}=(i_{1}\lc i_{k})$ such that $i_{a}\in\{1\lc m\}$, $a=1\lc k$.
We are going to construct a morphism $\mathcal{C}_{k}^{\bar{i}, m}\in\Hom_{\mathcal{D}_{t}}{(\C, \g ^{\otimes m})}$, which will play the role of the element 
\begin{equation}\label{C_{k}}
 C_{k}^{\bar{i}, m}=\sum_{j_{1}, j_{2}\lc j_{k}} e_{j_{1}j_{2}}^{(i_{1})} e_{j_{2}j_{3}}^{(i_{2})}\dots e_{j_{k}j_{1}}^{(i_{k})}\in \gl_{n}^{\otimes m}.
 \end{equation}

First, assume that all $i_{1}\lc i_{k}$ are distinct. In this case, we should have $k\leq m$. Define a permutation $\sigma\in S_{m}$ by $\sigma (a)=i_{a}$, $a=1\lc k$, and $\sigma (a)<\sigma (b)$ for all $k<a<b\leq m$. 
Since $\C[S_{m}]$ is naturally isomorphic to a subalgebra of $\End_{\mathcal{D}_{t}} (\g^{\otimes m})$, we identify $\sigma$ with the corresponding morphism in $\End_{\mathcal{D}_{t}} (\g^{\otimes m})$. Then we define $\mathcal{C}_{k}^{\bar{i},m}$ as follows
\[\mathcal{C}_{k}^{\bar{i},m}=\sigma\circ (\mathcal{C}_{k}\otimes \mathcal{C}_{1}^{\otimes(m-k)}).\]

The construction of $\mathcal{C}_{k}^{\bar{i},m}$ for any $\bar{i}$ should be clear from the following example.

Let $\bar{i}=(2,1,3,1,1)$. Then we define another sequence $\bar{j}=(4,1,5,2,3)$. The sequence $\bar{j}$ is obtained from $\bar{i}$ as follows: we consider subsequence $(1,1,1)$ of $\bar{i}$ and make it $(1,2,3)$, then we add $2$ to all other elements of $\bar{i}$. We define $\mathcal{C}_{5}^{\bar{i},3}$ as the composition of $\mathcal{C}_{5}^{\bar{j},5}$ and the morphism
\begin{center}
$\vcenter{\hbox{
\begin{tikzpicture}[
V/.style={fill,circle,inner sep=0pt,minimum size=2mm},
V*/.style={draw, circle,inner sep=0pt,minimum size=2mm},
]

\node[V] (1) at (0,0) {};
\node[V*] (2) at (1,0) {};
\node[V] (3) at (2,0) {};
\node[V*] (4) at (3,0) {};
\node[V] (5) at (4,0) {};
\node[V*] (6) at (5,0) {};

\node[V] (1a) at (0,-1) {};
\node[V*] (6a) at (5,-1) {};

\draw[bend right=60,-] (2.south) to (3.south);
\draw[bend right=60,-] (4.south) to (5.south);
\draw[-] (1.south) to (1a.north);
\draw[-] (6.south) to (6a.north);

 \end{tikzpicture}}}
\quad \otimes\quad \id _{\g^{\otimes 2}}\quad \in \Hom_{\mathcal{D}_{t}} (\g^{\otimes 5},\g^{\otimes 3})$.
\end{center}

\vspace{0.5cm}

Recall the tensor product $\bar{V}=V_{(\lambda^{(1)},\mu^{(1)})}\otimes V_{(\lambda^{(2)},\mu^{(2)})}\otimes\dots \otimes V_{(\lambda^{(m)},\mu^{(m)})}$ from Section \ref{prelim Deligne 3}.
We will now construct a morphism $\rho\in\Hom_{\mathcal{D}_{t}} (\g^{\otimes m}\otimes \bar{V},\bar{V})$, which will play the role of the action of  $U(\gl_{n})^{\otimes m}$ on a tensor product of $m$ finite dimensional irreducible $U(\gl_{n})$-modules.

Consider a tensor product $\bar{W}=W_{1}\otimes\dots\otimes W_{m}$, where $W_{a}=T^{r_{a},s_{a}}$, $r_{a}=|\lambda^{(a)}|$, and $s_{a}=|\mu^{(a)}|$, $a=1\lc m$. We will first construct a morphism $\widetilde{\rho}\in\Hom_{\mathcal{D}_{t}} (\g^{\otimes m}\otimes \bar{W},\bar{W})$.

Notice that $\bar{W}=A_{1}\otimes \dots\otimes A_{d}$, where $d=\sum_{a=1}^{m}(r_{a}+s_{a})$, and for all $i=1\lc d$, $A_{i}$ is either $V$ or $V^{*}$.
By pairing of the $i$-th tensor factor in $\g^{\otimes m}$ with the tensor factor $A_{j}$ in $\bar{W}$, we will mean the following two edges in diagrams describing $\widetilde{\rho}$:

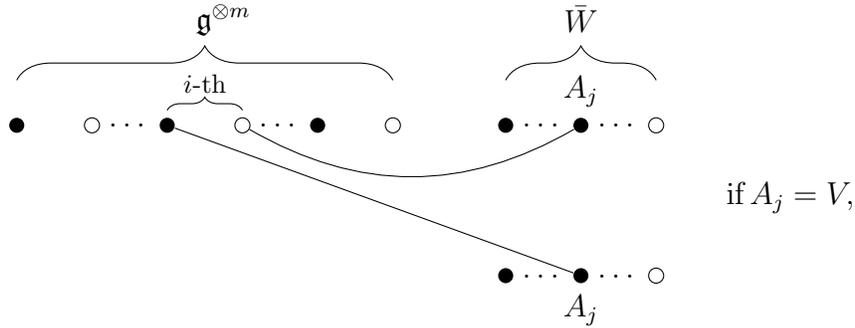
\[
\begin{tikzpicture}[baseline=-30pt,
V/.style={fill,circle,inner sep=0pt,minimum size=2mm},
V*/.style={draw, circle,inner sep=0pt,minimum size=2mm},
]

\node[V] (1) at (0,0) {};
\node[V*] (2) at (1,0) {};
\node[] (2dots3) at (1.5,0) {$\dots$};
\node[V] (3) at (2,0) {};
\node[V*] (4) at (3,0) {};
\node[] (4dots5) at (3.5,0) {$\dots$};
\node[V] (5) at (4,0) {};
\node[V*] (6) at (5,0) {};
\node[V] (7) at (6.5,0) {};
\node[] (7dots8) at (7,0) {$\dots$};
\node[V, label=above:{$A_{j}$}] (8) at (7.5,0) {};
\node[] (8dots9) at (8,0) {$\dots$};
\node[V*] (9) at (8.5,0) {};

\node[V] (7a) at (6.5,-2) {};
\node[] (7adots8a) at (7,-2) {$\dots$};
\node[V, label=below:{$A_{j}$}] (8a) at (7.5,-2) {};
\node[] (8adots9a) at (8,-2) {$\dots$};
\node[V*] (9a) at (8.5,-2) {};

\draw[bend right=30,-] (4) to (8);
\draw[-] (3) to (8a);

\draw [decorate,decoration={brace,amplitude=5pt}]
(2,0.2) -- (3,0.2)node [black,midway, yshift=10pt] {\footnotesize
$i$-th};
\draw [decorate,decoration={brace,amplitude=13pt}]
(0,0.6) -- (5,0.6)node [black,midway, xshift=7pt, yshift=23pt] {
$\g^{\otimes m}$};
\draw [decorate,decoration={brace,amplitude=13pt}]
(6.5,0.6) -- (8.5,0.6)node [black,midway, yshift=23pt] {
$\bar{W}$};
\end{tikzpicture}
\quad\quad\text{if}\, A_{j}=V,
\]
and
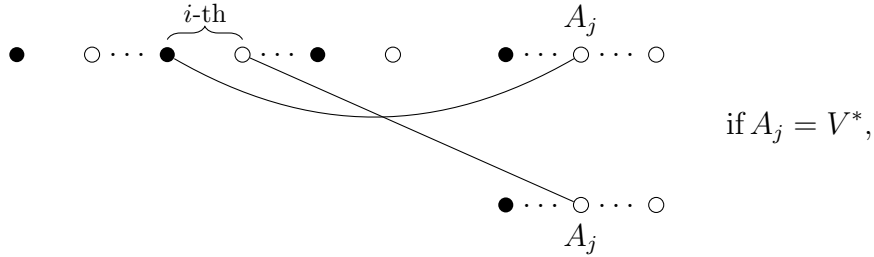
\[
\begin{tikzpicture}[baseline=-30pt,
V/.style={fill,circle,inner sep=0pt,minimum size=2mm},
V*/.style={draw, circle,inner sep=0pt,minimum size=2mm},
]

\node[V] (1) at (0,0) {};
\node[V*] (2) at (1,0) {};
\node[] (2dots3) at (1.5,0) {$\dots$};
\node[V] (3) at (2,0) {};
\node[V*] (4) at (3,0) {};
\node[] (4dots5) at (3.5,0) {$\dots$};
\node[V] (5) at (4,0) {};
\node[V*] (6) at (5,0) {};
\node[V] (7) at (6.5,0) {};
\node[] (7dots8) at (7,0) {$\dots$};
\node[V*, label=above:{$A_{j}$}] (8) at (7.5,0) {};
\node[] (8dots9) at (8,0) {$\dots$};
\node[V*] (9) at (8.5,0) {};

\node[V] (7a) at (6.5,-2) {};
\node[] (7adots8a) at (7,-2) {$\dots$};
\node[V*, label=below:{$A_{j}$}] (8a) at (7.5,-2) {};
\node[] (8adots9a) at (8,-2) {$\dots$};
\node[V*] (9a) at (8.5,-2) {};

\draw[bend right=30,-] (3) to (8);
\draw[-] (4) to (8a);

\draw [decorate,decoration={brace,amplitude=5pt}]
(2,0.2) -- (3,0.2)node [black,midway, yshift=10pt] {\footnotesize
$i$-th};
\end{tikzpicture}
\quad\quad\text{if}\, A_{j}=V^{*},
\]

Then we define $\widetilde{\rho}$ as the sum of all diagrams such that:
\begin{itemize}
    \item the first tensor factor in $\g^{\otimes m}$ is paired with one of $A_{1}\lc A_{r_{1}+s_{1}}$ (the factors in $W_{1}=T^{r_{1},s_{1}}$),
    \item the second tensor factor in $\g^{\otimes m}$ is paired with one of $A_{r_{1}+s_{1}+1}\lc A_{r_{2}+s_{2}}$ (the factors in $W_{2}=T^{r_{2},s_{2}}$), and so on;
    \item for all $A_{i}$ in $\bar{W}$, which are not paired with any tensor factors in $\g^{\otimes m}$, we just connect the corresponding vertex in the top row to the $i$-th vertex in the bottom row ("$A_{i}$ goes to itself").
\end{itemize}

We define 
\[\rho\coloneqq(e_{(\lambda^{(1)},\mu^{(1)})}\otimes\dots\otimes e_{(\lambda^{(m)},\mu^{(m)})})\circ\widetilde{\rho}\circ (\id_{\mathfrak{g}^{\otimes m}}\otimes e_{(\lambda^{(1)},\mu^{(1)})}\otimes\dots\otimes e_{(\lambda^{(m)},\mu^{(m)})}).\]

Let $e$ be a morphism in $\Hom_{\mathcal{D}_{t}} (\C, \g^{\otimes m})$. Then $e\otimes \id_{\bar{V}}\in\Hom_{\mathcal{D}_{t}} (\C\otimes\bar{V}, \g^{\otimes m}\otimes\bar{V})=\Hom_{\mathcal{D}_{t}} (\bar{V}, \g^{\otimes m}\otimes\bar{V})$.
Denote $e(\bar{V})\coloneqq \rho\circ (e\otimes\id_{\bar{V}})\in\End_{\mathcal{D}_{t}}(\bar{V})$.

Recall the generators $S_{klj}^{(r)}$ of the Bethe algebra $\mathcal{B}_{n}$ from Section \ref{Bethe 1}. They depend on the sequence $\bar{z}=(z_{1}\lc z_{m})$ of distinct complex numbers. It is easy to see that for each $k$, $l$, $j$, and $r$, we have 
\[ S_{klj}^{(r)}=\sum_{l'<l}\sum_{\bar{i}} a^{\bar{i}}_{l'}(k,l,j,r,\bar{z}) C^{\bar{i}, m}_{l'},\]
where $a^{\bar{i}}_{l'}(k,l,j,r,\bar{z})$ are some complex numbers, which do not depend on $n$. 

We define 
\[ \mathcal{S}_{klj}^{(r)}=\sum_{l'<l}\sum_{\bar{i}} a^{\bar{i}}_{l'}(k,l,j,r,\bar{z}) \mathcal{C}^{\bar{i}, m}_{l'}(\bar{V}).\]
Let us write $\bar{\lambda}$ (resp., $\bar{\mu}$) for the sequece of partitions $(\lambda_{1}\lc\lambda_{m})$ (resp., $(\mu_{1}\lc\mu_{m})$). Define the Bethe algebra $\underline{\mathcal{B}}_{t}(\bar{\lambda},\bar{\mu})$ in the Deligne's category to be the unital subalgebra of $\End_{\mathcal{D}_{t}}(\bar{V})$ generated by $\mathcal{S}_{klj}^{(r)}$, $k\geq 1$, $1\leq l\leq k$, $1\leq j\leq l$, $1\leq r\leq m$.

Fix $n\in\Z_{>0}$, $n\geq l_{\lambda}^{(a)}+l_{\mu}^{(a)}$, $a=1\lc m$, where $l_{\lambda}^{(a)}$ (resp., $l_{\mu}^{(a)}$) is the number of nonzero elements  in the partition $\lambda^{(a)}$ (resp., $\mu^{(a)}$). For each $a=1\lc m$, let $\nu^{(a)}$ be a $\gl_{n}$-weight defined by $\lambda^{(a)}$ and $\mu^{(a)}$ according to formula \eqref{lambda mu nu}. Denote $\bar{\nu}=(\nu^{(1)}\lc\nu^{(m)})$.
Recall the homomorphism $G_{n}^{\bar{V}}:\End_{\mathcal{D}_{n}}(\bar{V})\rightarrow \End_{\mathcal{R}_{n}}(V(\bar{\nu}))$. The algebra $\underline{\mathcal{B}}_{t}(\bar{\lambda},\bar{\mu})$ is constructed in such a way that $G_{n}^{\bar{V}}(\underline{\mathcal{B}}_{t}(\bar{\lambda},\bar{\mu}))=\mathcal{B}_{n}(\bar{\nu})$.

Recall that $N=\sum_{a=1}^{m}(|\lambda^{(a)}|+|\mu^{(a)}|)$. The following lemma is an immediate consequence of Lemma \ref{G_n inj}.

\begin{lem}\label{B is B}
For each $n\in\Z_{\geq N}$, we have
\[\underline{\mathcal{B}}_{\,n}(\bar{\lambda}, \bar{\mu})\cong \mathcal{B}_{n}(\bar{\nu}).\]
\end{lem}

\subsection{}
Denote $I_{1}= \{a\,\vert\,\mu^{(a)}=\emptyset\}$ and $I_{2}= \{a\,\vert\,\lambda^{(a)}=\emptyset\}$. For the rest of the paper, assume that 
\begin{equation}\label{condition on bipartitions}
    \{1\lc m\}=I_{1}\sqcup I_{2}.
\end{equation} Then, as was mentioned in Section \ref{prelim Deligne 3}, the idempotents $e_{(\lambda^{(a)},\,\mu^{(a)})}$ do not depend on $t$, and it is easy to define an analog $\mathcal{B}_{[w]}(\bar{\lambda},\bar{\mu})$ of the algebra $\underline{\mathcal{B}}_{t}(\bar{\lambda},\bar{\mu})$, where the role of the complex number $t$ is played by a variable $w$. 
We construct the algebra $\mathcal{B}_{[w]}(\bar{\lambda},\bar{\mu})$ using the procedure for the construction of $\underline{\mathcal{B}}_{t}(\bar{\lambda},\bar{\mu})$ with the following changes:
\begin{itemize}
    \item Any vector space $M$ involved in the construction of $\underline{\mathcal{B}}_{t}(\bar{\lambda},\bar{\mu})$ replace by the vector space $M\otimes\C [w]$.
    \item The rule "erasing a loop = multiplication by $t$" replace by "erasing a loop = multiplication by $w$".
\end{itemize}
By construction, $\mathcal{B}_{[w]}(\bar{\lambda},\bar{\mu})$ is a finitely-generated free $C[w]$-module, and for all $t\in\C$, we have
\begin{equation}\label{B eval}
    \mathcal{B}_{[w]}(\bar{\lambda},\bar{\mu})/(w-t)\cong\underline{\mathcal{B}}_{t}(\bar{\lambda},\bar{\mu}).
\end{equation}
\begin{rem}\label{R1}
If we don't have condition \eqref{condition on bipartitions}, then the dependence of $e_{(\lambda^{(a)},\,\mu^{(a)})}$ on $t$ is not polynomial (see Example 4.3.1 in \cite{CoWi}), and there are no analogs of these idempotents in $T^{|\lambda^{(a)}|,|\mu^{(a)}|}\otimes \C [w]$.
\end{rem}
\begin{lem}\label{1.1}
Let $\pi_{n}$ denote the projection $\mathcal{B}_{[w]}(\bar{\lambda},\bar{\mu})\rightarrow \mathcal{B}_{[w]}(\bar{\lambda},\bar{\mu})/(w-t)$. Suppose that for some $p\in \mathcal{B}_{[w]}(\bar{\lambda},\bar{\mu})$ and $N_{0}\in\Z_{>0}$, we have $\pi_{n}(p)=0$, $n\in\Z_{\geq N_{0}}$. Then $p=0$.  
\end{lem}
\begin{proof}
Let $f_{1}\lc f_{k}$ be a basis of $\mathcal{B}_{[w]}(\bar{\lambda},\bar{\mu})$ as a $\C [w]$-module. Then $p=p_{1}(w)f_{1}+p_{2}(w)f_{2}+\dots +p_{k}(w)f_{k}$ for some $p_{1}(w)\lc p_{k}(w)\in\C [w]$. 

Take $n\in\Z_{\geq N_{0}}$. Since $\pi_{n}(p)=0$, we have $p=(w-n)q$ for some $q=q_{1}(w)f_{1}+\dots +q_{k}(w)f_{k}$, $q_{1}(w)\lc q_{k}(w)\in\C [w]$. Since $f_{1}\lc f_{k}$ are linearly independent, we have $p_{i}(w)=(w-n)q_{i}(w)$, $i=1\lc k$, so $p_{i}(n)=0$, $i=1\lc k$. 

It follows that polynomials $p_{1}(w)\lc p_{k}(w)$ become zero if we evaluate them at any point of an infinite subset of $\C$. Therefore, $p_{i}(w)=0$, $i=1\lc k$, which means $p=0$.
\end{proof}
\begin{prop}
For all $t\in\C$, the algebra $\underline{\mathcal{B}}_{\,t}(\bar{\lambda}, \bar{\mu})$ is commutative.
\end{prop}
\begin{proof}
For any two elements $a_{1}$,  $a_{2}$ of the algebra $\mathcal{B}_{[w]}(\bar{\lambda}, \bar{\mu})$, consider its commutator $p=[a_{1},a_{2}]$. By Lemma \ref{B is B}, for any $n\in\Z_{\geq N}$, the algebra $\underline{\mathcal{B}}_{\,n}(\bar{\lambda}, \bar{\mu})$ is commutative. Therefore, $\pi_{n}(p)=0$, $n\in\Z_{\geq N}$. By Lemma \ref{1.1}, this implies $p=0$, so $\mathcal{B}_{[w]}(\bar{\lambda}, \bar{\mu})$ is commutative, and the proposition follows from formula  \eqref{B eval}. 
\end{proof}

\subsection{}
Recall the generators $B_{ij}^{(a)}$ of the Bethe algebra $\mathcal{B}_{n}$ introduced in Section \ref{Bethe 2}. Denote by $B_{ij}^{(a)}(\bar{\nu})$ the image of $B_{ij}^{(a)}$ under the projection $\mathcal{B}_{n}\rightarrow \mathcal{B}_{n}(\bar{\nu})$. 
Let us now define elements $\underline{B}^{(a)}_{ij}(\bar{\lambda},\bar{\mu})$ of the algebra $\underline{\mathcal{B}}_{t}(\bar{\lambda},\bar{\mu})$, which will be analogs of the generators $B_{ij}^{(a)}(\bar{\nu})$.

Lemma \ref{Newton} implies that for each $r\in\Z_{> 0}$, $s=1\lc r$, $j=1\lc s$, $a=1\lc m$, there exist polynomials $P_{rsj}^{(a)}$ in variables $x_{r's'j'}^{(a')}$, $r'=1\lc r$, $s'=1\lc r'$, $j=1\lc s'$, $a'=1\lc m$ such that $P_{rsj}^{(a)}$ do not depend on $n$, and for every $n\geq \max_{b} (l_{\lambda}^{(b)}+l_{\mu}^{(b)})$ and $r\leq n$, we have
\[ P_{rsj}^{(a)}\bigl(x_{r's'j'}^{(a')}=S_{r's'j'}^{(a')}\bigr)=\tilde{B}_{rsj}^{(a)}.\]

Define
\[\underline{\tilde{B}}_{rsj}^{(a)}(\bar{\lambda},\bar{\mu})\coloneqq P_{rsj}^{(a)}(x_{r's'j'}^{(a')}=\mathcal{S}_{r's'j'}^{(a')}),\quad  \underline{B}^{(a)}_{ij}(\bar{\lambda},\bar{\mu}) \coloneqq \underline{\tilde{B}}_{i,i,j}^{(a)}(\bar{\lambda},\bar{\mu}) \]
It is immediate that the elements $\underline{\tilde{B}}_{rsj}^{(a)}(\bar{\lambda},\bar{\mu})$ generate the algebra $\underline{\mathcal{B}}_{t}(\bar{\lambda},\bar{\mu})$. 

\begin{lem}
The elements $\underline{B}^{(a)}_{ij}(\bar{\lambda},\bar{\mu})$ generate the algebra $\underline{\mathcal{B}}_{t}(\bar{\lambda},\bar{\mu})$.
\end{lem}
\begin{proof}
For each $r\in\Z_{> 0}$, $s=1\lc r$, define a rational function $\underline{\tilde{B}}_{rs}(u)$ as follows
\begin{equation}\label{underline B(u)}
\underline{\tilde{B}}_{rs}(u)=\sum_{a=1}^{m}\sum_{j=1}^{s}\frac{\underline{\tilde{B}}_{rsj}^{(a)}}{(u-z_{a})^{j}}.
\end{equation}
By formula \eqref{B relation}, Lemma \ref{B is B}, and Lemma \ref{1.1}, we have
\begin{equation}\label{underline B relation}
\underline{\tilde{B}}_{rs}(u)=\underline{\tilde{B}}_{ss}(u)\binom{t-s}{r-s},
\end{equation}
which proves the lemma.
\end{proof}
 
\section{Fuchsian differential operators with no monodromy}\label{5}
\subsection{}
Let $\mathcal{A}$ be a commutative algebra. Consider the algebra $\mathcal{A}((x))$ of Laurent series in $x$ with coefficients in $\mathcal{A}$. Denote by $\partial_{x}$ the derivative with respect to $x$. 

Fix $n$ integers $m_{1}<\dots <m_{n}$.  Consider a differential operator
\begin{equation}\label{D}
D=\sum_{i=0}^{n}\sum_{j=-i}^{\infty}b_{ij}x^{j}\partial^{n-i}_{x}
\end{equation}
with some $b_{ij}\in\mathcal{A}$. We are going to give a necessary and sufficient condition for a differential equation $Df=0$ to have $n$ solutions $f_{1}\lc f_{n}\in\mathcal{A}((x))$ of the form 
\begin{equation}\label{solution}
f_{i}=x^{m_{i}}+a_{1}^{(i)}x^{m_{i}+1}+a_{2}^{(i)}x^{m_{i}+2}+\dots .
\end{equation}
This condition will be in the form of certain polynomial relations on the coefficients $b_{ij}$ of $D$.

For each $k=0,1,2,\dots$ and $\alpha\in\C$, define 
\[
r_{k}(\alpha)=\sum_{i=0}^{n}b_{i,-i+k}\alpha^{\underline{n-i}},
\]
where $\alpha^{\underline{j}}=\alpha(\alpha-1)(\alpha-2)\dots (\alpha-j+1)$, $j\in\Z_{\geq 0}$.

Notice that the coefficient for the lowest degree of $x$ in $Df_{i}\in\mathcal{A}((x))$ equals $r_{0}(m_{i})$. Therefore, we get a necessary condition for $D$ to have a solution of the form \eqref{solution}:
\begin{equation}\label{residue condition}
    r_{0}(m_{i})=0,\quad i=1\lc n.
\end{equation}

For each $1\leq i<j\leq n$, consider a matrix 
\begin{equation}\label{Aij}
A_{ij}=
\begin{tikzpicture}[baseline=(current bounding box.center)]
\matrix (m) [matrix of math nodes,nodes in empty cells,right delimiter={)},left delimiter={(} ]
{
r_{1}(m_{i}) & r_{2}(m_{i}) & & & r_{m_{j}-m_{i}}(m_{i})\\
r_{0}(m_{i}+1) & r_{1}(m_{i}+1) & & & \\
0 & r_{0}(m_{i}+2) & & &  \\
&  & & & r_{2}(m_{j}-2) \\
0 & & 0\,\,\, & r_{0}(m_{j}-1) & r_{1}(m_{j}-1)\\
} ;
\draw[loosely dotted, thick] ($(m-2-2)+(0.8,-0.3)$)-- ($(m-5-5)+(-1.1,0.2)$);
\draw[loosely dotted, thick] (m-3-2)-- (m-5-4);
\draw[loosely dotted, thick] (m-1-2)-- ($(m-4-5)+(-1.1,0.2)$);
\draw[loosely dotted, thick] (m-3-1)-- (m-5-1);
\draw[loosely dotted, thick] (m-5-1)-- (m-5-3);
\draw[loosely dotted, thick] (m-1-2)-- (m-1-5);
\draw[loosely dotted, thick] (m-1-5)-- (m-4-5);
\draw[loosely dotted, thick] (m-3-1)-- (m-5-3);
\end{tikzpicture}.
\end{equation}
Define a matrix $\bar{A}_{ij}$ by crossing out columns and rows of $A_{ij}$ that contain elements $r_{0}(m_{l})$, $i<l<j$. 
\begin{prop}\label{prop1}
The differential equation $Df=0$ has solutions $f_{1}\lc f_{n}$ of the form \eqref{solution} if and only if \eqref{residue condition} holds and 
\begin{equation}\label{monodromy condition}
\det \bar{A}_{ij} =0,\quad \quad 1\leq i<j\leq n.
\end{equation}
\end{prop}
\begin{proof}
Suppose that we have solutions $f_{1}\lc f_{n}$ of the form \eqref{solution}. Then, as noted above, condition \eqref{residue condition} holds. We may assume that for coefficients $a^{(i)}_{j}$ of $f_{i}$, see formula \eqref{solution}, we have $a^{(i)}_{m_{k}-m_{i}}=0$ for all $k>i$.

Equating the coefficients for $x^{m_{i}+l-n}$ on the both sides of the relation $Df_{i}=0$, we have 
\begin{equation}\label{recursion}
\sum_{j=0}^{l-1}a^{(i)}_{j}r_{l-j}(m_{i}+j)+a^{(i)}_{l}r_{0}(m_{i}+l)=0.
\end{equation}

For each $l=1,2\lc m_{i}-m_{j}$ such that $m_{i}+l\neq m_{k}$, $k>i$, denote by $\bar{P}_{il}$ the principal minor of the matrix $\bar{A}_{ij}$ with the last row $\{0\lc 0, r_{0}(m_{i}+l), r_{1}(m_{i}+l)\}$. Using the expansion of $\bar{P}_{il}$ along the last column and induction on $l$, one checks that \eqref{recursion} implies 
\begin{equation}\label{coefficients}
a^{(i)}_{l}=(-1)^{l}\bar{P}_{il}\left(\underset{k=1}{\overset{l}{{\prod}'}}r_{0}(m_{i}+k)\right)^{-1},
\end{equation}
where $\underset{k=1}{\overset{l}{{\prod}'}}$ denotes the product over $k=1\lc l$ such that $m_{i}+k\neq m_{j}$ for all $j>i$. Notice that because of condition \eqref{residue condition}, we have $r_{0}(\alpha)=\prod_{i=1}^{n}(\alpha - m_{i})$ for any $\alpha\in\C$, in particular $\underset{k=1}{\overset{l}{{\prod}'}}r_{0}(m_{i}+k)\in\C^{\times}$, and the right hand side of formula \eqref{coefficients} is well-defined.

Now, take $l$ such that $m_{i}+l=m_{k}$ for some $k>i$ in formula \eqref{recursion}. Because of condition \eqref{residue condition}, we have $r_{0}(m_{i}+l)=0$, and formula \eqref{recursion} reads for such $l$ as
\begin{equation}\label{recursion degeneration}
\sum_{j=0}^{l-1}a^{(i)}_{j}r_{l-j}(m_{i}+j)=0.
\end{equation}

But because of \eqref{coefficients}, the left hand side of \eqref{recursion degeneration} multiplied by $\underset{j=1}{\overset{l-1}{{\prod}'}}r_{0}(m_{i}+j)$ is the expansion of $\det \bar{A}_{ik}$ along the last column. This proves the "only if" part of the proposition.

Conversely, assume that conditions \eqref{residue condition} and \eqref{monodromy condition} hold. Fix $i=1\lc n$. Then for each $l<m_{n}-m_{i}$ such that $m_{i}+l\neq m_{k}$ for all $k>i$, we define $a^{(i)}_{l}$ as in formula \eqref{coefficients}. If $m_{i}+l=m_{k}$ for some $k>i$, then we put $a^{(i)}_{l}=0$. If $j>n$, then we define $m_{j}\coloneqq m_{n}+j-n$. This allows us define $a^{(i)}_{l}$ for $l>m_{n}-m_{i}$ similarly to the case $l<m_{n}-m_{i}$. 

With $a^{(i)}_{l}$ defined in such a way, relation \eqref{recursion} hold for every $l=1,2,\dots $ (condition \eqref{monodromy condition} ensures that \eqref{recursion} is true when $m_{i}+l=m_{k}$ for some $k$). Therefore, the series $f_{i}$ defined by formula \eqref{solution} solves the differential equation $Df=0$. This proves the "if" part of the proposition.
\end{proof}

\subsection{}\label{5.2}
Let $\mathcal{A}(u)=\mathcal{A}\otimes\C (u)$ be the algebra of $\mathcal{A}$-valued rational functions of $u$. Let $ \mathcal{D}\bigl(\mathcal{A}((x))\bigr)$ and $\mathcal{D}\bigl(\mathcal{A}(u)\bigr)$ be the algebras of differential operators with coefficients in $\mathcal{A}((x))$ and $\mathcal{A}(u)$, respectively. For any $z\in\C$, define a map
\begin{equation}\label{Lz}
\begin{split}
    L_{z}:\quad & \mathcal{D}\bigl(\mathcal{A}(u)\bigr) \rightarrow  \mathcal{D}\bigl(\mathcal{A}((x))\bigr),\\
    &\sum_{i=1}^{n}b_{i}(u)\partial_{u}^{n-i}  \mapsto  \sum_{i=1}^{n}L_{z}\bigl[b_{i}(u)\bigr](x)\partial_{x}^{n-i}.
\end{split}
\end{equation}
Here, $L_{z}\bigl[b_{i}(u)\bigr](x)$ is the Laurent series of $b_{i}(u)$ at $z$, where we take $u-z=x$.

We will also need a map $L_{\infty}:\, \mathcal{D}\bigl(\mathcal{A}(u)\bigr) \rightarrow  \mathcal{D}\bigl(\mathcal{A}((x))\bigr)$ defined by 
\[
L_{\infty}:\,\sum_{i=1}^{n}b_{i}(u)\partial_{u}^{n-i}  \mapsto  \left(-\frac{1}{x^{2}}\right)^{n}\sum_{i=1}^{n}L_{0}\bigl[b_{i}(u^{-1})\bigr](x)(-x^{2}\partial_{x})^{n-i}.
\]

Now, let us take $\mathcal{A}=\mathcal{A}_{n}$, where $\mathcal{A}_{n}$ is the commutative algebra with free generators $y_{ij}^{(a)}$, $i=1\lc n$, $j=1\lc i$, $a=1\lc m$. Let $\bar{z}=(z_{1}\lc z_{m})$ be a sequence of pairwise distinct complex numbers, and let $\bar{\nu}=(\nu^{(1)},\nu^{(2)}\lc\nu^{(m)})$ be a sequence of dominant integral $\gl_{n}$-weights. Consider a differential operator $D^{\C [y]}\in\mathcal{D}\bigl(\mathcal{A}_{n}(u)\bigr)$ defined by the formula 
\begin{equation}\label{Dcy}
D^{\C [y]}=\partial_{u}^{n}+\sum_{i=1}^{n}\sum_{j=1}^{i}\sum_{a=1}^{m}\frac{y_{ij}^{(a)}}{(u-z_{a})^{j}}\partial_{u}^{n-i}.
\end{equation}

Recall the polynomials $r_{k}(\alpha)$ and matrices $\bar{A}_{ij}$ that we constructed for a differential operator $D\in\mathcal{D}\bigl(\mathcal{A}((x))\bigl)$ of the form \eqref{D}. The construction of $\bar{A}_{ij}$ depends on numbers $m_{1}<m_{2}<\dots <m_{n}$. Notice that for each $a=1\lc m$, the differential operator $L_{z_{a}}\bigl(D^{\C [y]}\bigr)$ is of the form \eqref{D}. Taking $D=L_{z_{a}}\bigl(D^{\C [y]}\bigr)$ and $m_{n+1-i}=n+\nu^{(a)}_{i}-i$, we denote the corresponding polynomials $r_{k}(\alpha)$ and matrices $\bar{A}_{ij}$ as $r^{(a)}_{k}(\alpha)$ and $\bar{A}^{(a)}_{ij}$, respectively.

The differential operator $L_{\infty}\bigl(D^{\C [y]}\bigr)$ is not of the form \eqref{D}, but there are elements $b_{ij}^{(\infty)}\in\mathcal{A}_{n}$ such that
\[L_{\infty}\bigl(D^{\C [y]}\bigr)=\sum_{i=0}^{n}\sum_{j=-2i+1}^{\infty}b_{ij}^{(\infty)}x^{j}\partial_{x}^{n-i}. \]
The difference with \eqref{D} is that the summation over $j$ starts from $-2i+1$ rather then $-i$. 

Let $J_{n}(\bar{\nu})$ be the ideal in $\mathcal{A}_{n}$ generated by the elements
\begin{equation}\label{gen1}
r_{0}^{(a)}(n+\nu^{(a)}_{i}-i),\quad a=1\lc m,\, i=1\lc n,
\end{equation}
\begin{equation}\label{gen2}
  \det \bar{A}_{ij}^{(a)},\quad a=1\lc m,  
\end{equation}
and
\begin{equation}\label{gen4}
    b_{ij}^{(\infty)}, \quad j=-2i+1,\, -2i+2\lc -i-1.
\end{equation}

Denote $\mathcal{F}_{n}=\mathcal{A}/J_{n}(\bar{\nu})$. Notice that by Proposition \ref{prop1}, the set of closed points of the scheme $\on{Spec} \mathcal{F}_{n}$ is the set of monodromy-free Fuchsian differential operators with regular singularities at $z_{1},z_{2}\lc z_{m}$ and $\infty$, with exponents at $z_{a}$ given by $n+\nu^{(a)}_{i}-i$, $i=1\lc n$. The algebra $\mathcal{F}_{n}$ is closely related to the projection $\mathcal{B}_{n}(\bar{\nu})$ of the Bethe algebra introduced above. Namely, we have the following theorem.

\begin{thm}\label{T1}
The kernel of the homomorphism
\begin{equation*}
    \begin{split}
\phi:\quad & \mathcal{A}_{n} \rightarrow \mathcal{B}_{n}(\bar{\nu}) \\
& y^{(a)}_{ij}\mapsto B_{ij}^{(a)}(\bar{\nu})
    \end{split}
\end{equation*}   
is the ideal $J_{n}(\nu)$.
\end{thm}

This theorem is one of the main results in the study of the Bethe ansatz method for the quantum Gaudin model. It is proved in \cite{Ryb}.

A result similar to Theorem \ref{T1}, can be found in work \cite{MTV8}, where it is proved that $\mathcal{B}_{n}(\bar{\nu})$ is isomorphic to the algebra of functions on an intersection of Schubert cells. Points of this intersection correspond to kernels of the Fuchsian differential operators from $\on{Spec} \mathcal{F}_{n}$. Considering $\mathcal{B}_{n}(\bar{\nu})$ as the algebra of functions on differential operators is more suitable for our goals since it is not clear how to generalize the intersection of Schubert cells to arbitrary complex $n$.

\section{The interpolation of the no-monodromy conditions}\label{6}
Recall that $\bar{\lambda}$ and $\bar{\mu}$ are sequences of partitions satisfying condition \eqref{condition on bipartitions}.
In this section, we are going to prove an analog of Theorem \ref{T1} for the algebra  $\underline{\mathcal{B}}_{t}(\bar{\lambda},\bar{\mu})$. 

For each $n\in\Z_{>0}$, let us now think of $\mathcal{A}_{n}$ as a subalgebra in the infinitely generated free commutative algebra $\mathcal{A}_{\infty}$ with free generators $y_{ij}^{(a)}$, $1\leq a\leq m$, $i\geq 1$, $1\leq j\leq i$. Recall the generators $\underline{B}_{ij}^{(a)}(\bar{\lambda},\bar{\mu})$, $1\leq a\leq m$, $i\geq 1$, $1\leq j\leq i$ of the algebra $\underline{\mathcal{B}}_{t}(\bar{\lambda},\bar{\mu})$. Our goal is to find the kernel of the homomorphism
\begin{equation*}
    \begin{split}
\underline{\phi}_{t}:\quad & \mathcal{A}_{\infty} \rightarrow \underline{\mathcal{B}}_{t}(\bar{\lambda},\bar{\mu}), \\
& y^{(a)}_{ij}\mapsto \underline{B}_{ij}^{(a)}(\bar{\lambda},\bar{\mu}).
    \end{split}
\end{equation*}  

Recall that $l_{\lambda}^{(a)}$ (resp., $l_{\mu}^{(a)}$) is the number of nonzero elements in the partition $\lambda^{(a)}$ (resp., $\mu^{(a)}$). Denote $l^{(a)} = l_{\lambda}^{(a)} + l_{\mu}^{(a)}$. 

Let $n$ be an integer such that $n\geq \max_{a} l^{(a)}$. For each $a=1\lc m$, let $\nu^{(a)}=(\nu^{(a)}_{1}\lc \nu^{(a)}_{n})$ be a $\gl_{n}$ weight defined by $\lambda^{(a)}$ and $\mu^{(a)}$ according to formula \eqref{lambda mu nu}.

To find the kernel of the map $\underline{\phi}_{t}$, we will first construct a generating set of the ideal $J_{n}(\bar{\nu})$ that "stabilizes for large $n$".

In this section, for convenience, we will sometimes drop the upper index $(a)$ assuming that it is equal to some fixed value, say, $a=1$, for example, $l\coloneqq l^{(1)}$, $r_{k}(\alpha)\coloneqq r_{k}^{(1)}(\alpha)$, $A_{ij}\coloneqq A^{(1)}_{ij}$, and so on.

Recall that 
\[ r_{k}(\alpha)=\sum_{i=0}^{n}b_{i,-i+k}\alpha^{\underline{n-i}},\]
where $\alpha^{\underline{j}}=\alpha(\alpha-1)(\alpha-2)\dots (\alpha-j+1)$ and $b_{i,-i+k}\in\mathcal{A}_{n}$ are the coefficients in the expansion
\[ \sum_{a=1}^{m}\sum_{j=1}^{i}\frac{y^{(a)}_{ij}}{(u-z_{a})^{j}}=b_{i,-i}(u-z_{1})^{-i}+b_{i,-i+1}(u-z_{1})^{-i+1}+b_{i,-i+2}(u-z_{1})^{-i+2}+\dots .\]

For each $i=0\lc n$, $j\geq -i$, define $\tilde{b}_{ij}\in\mathcal{A}_{n}$ by 
\begin{equation*}\label{b_i tilde}
r_{k}(\alpha)=\sum_{i=0}^{n}\tilde{b}_{i,-i+k}(\alpha-l_{\mu})^{\underline{n-i}},
\end{equation*}
and consider polynomials $\tilde{r}_{k}(\alpha)$ given by
\begin{equation*}
    \tilde{r}_{k}(\alpha) = \sum_{i=0}^{l+k} \tilde{b}_{i,-i+k}(\alpha-n+l_{\lambda}+k)^{\underline{l+k-i}},
\end{equation*}
where we assume that $\tilde{b}_{ij}=0$ if $i>n$.

Notice that 
\begin{align}
 & (\alpha-l_{\mu})^{\underline{n-l-k}}\tilde{r}_{k}(\alpha)=r_{k}(\alpha)-\sum_{i=l+k+1}^{n} b_{i,-i+k}(\alpha-l_{\mu})^{\underline{n-i}},\quad & \text{if}\quad k+l\leq n, \label{4.2} \\ 
 & \tilde{r}_{k}(\alpha)=(\alpha-n+l_{\lambda}+k)^{\underline{l+k-n}}r_{k}(\alpha),\quad & \text{if}\quad k+l\geq n. \label{4.3}
 \end{align}
 
For each $1\leq i<j\leq n$, let $M_{ij}$ be the matrix constructed from $A_{ij}$ by replacing each entry $r_{s}(\alpha)$ with $(\alpha - n+l_{\lambda}+k-d)^{\underline{k-d-s}}\tilde{r}_{s}(\alpha)$, where $k=m_{j}-m_{i}$, and $d+1$ is the number of the row containing the entry. Recall that the matrix $\bar{A}_{ij}$ was constructed from $A_{ij}$ by deleting certain rows and columns. Let $\bar{M}_{ij}$ be the matrix constructed from $M_{ij}$ in a similar way.

We will need the following two lemmas.

\begin{lem}\label{useful lemma 1}
Suppose that $l_{\lambda}\geq 1$. Then for each $j=1\lc l_{\lambda}$, $j<i\leq n$, we have 
\begin{equation}\label{4.1.1}
    r_{0}(m_{n+1-j})=c_{1}\tilde{r}_{0}(m_{n+1-j})+M_{1},
\end{equation}
\begin{equation}\label{4.1.2}
    \det \bar{A}_{n+1-i,n+1-j}= c_{2} \det \bar{M}_{n+1-i,n+1-j}+M_{2},
\end{equation}
where $c_{1}$, $c_{2}$ are nonzero complex numbers, and $M_{1}$, $M_{2}$ belong to the ideal in $\mathcal{A}_{n}$ generated by $\tilde{b}_{q,-q+s}$, $s\in \Z_{\geq 0}$, $q>l+s$.
\end{lem}
\begin{proof}
Since $l_{\lambda}\geq 1$, by condition \eqref{condition on bipartitions}, we have $l_{\mu}=0$ and $l=l_{\lambda}$.
Since $l\leq n$, one can use formula \eqref{4.2} for $k=0$ to see that 
\[ r_{0}(m_{n+1-j}) = (m_{n+1-j})^{\underline{n-l}}\tilde{r}_{0}(m_{n+1-j})+M_{1},\]
where $M_{1}$ belongs to the ideal in $\mathcal{A}_{n}$ generated by $\tilde{b}_{q,-q+s}$, $s\in \Z_{\geq 0}$, $q>l+s$.
For $j\leq l$, we have  $m_{n+1-j} = \lambda_{j}+n-j>n-l$, therefore $(m_{n+1-j})^{\underline{n-l}}\neq 0$, and formula \eqref{4.1.1} is proved.

To prove formula \eqref{4.1.2}, put $k=m_{n+1-j}-m_{n+1-i}$ and fix some $d=0\lc k-1$. First, suppose that $k-d+l\leq n$. Using \eqref{4.2}, one can see that up to summation with linear combinations of the elements $\tilde{b}_{q,-q+s}$, $s\in \Z_{\geq 0}$, $q>l+s$, the $(d+1)$-th row of $A_{n+1-i,n+1-j}$ is the $(d+1)$-th row of $M_{n+1-i,n+1-j}$ multiplied by $(m_{n+1-i}+d)^{\underline{n-l-k+d}}$. For $j\leq l$, we have $m_{n+1-i}=m_{n+1-j}-k=\lambda_{j}+n-j-k > n-l-k$. Therefore, $(m_{n+1-i}+d)^{\underline{n-l-k+d}}\neq 0$. 

If $k-d+l>n$, then, using \eqref{4.3}, one can see that up to summation with linear combinations of the elements $\tilde{b}_{q,-q+s}$, $s\in \Z_{\geq 0}$, $q>l+s$, the $(d+1)$-th row of $M_{n+1-i,n+1-j}$ is the $(d+1)$-th row of $A_{n+1-i,n+1-j}$ multiplied by $(m_{n+1-i}-n+l+k)^{\underline{k-d+l-n}}$. Since $l_{\mu}=0$, we have $m_{n+1-i}=\lambda_{n+1-i}+n-i\geq 0$, and $(m_{n+1-i}-n+l+k)^{\underline{k-d+l-n}}\neq 0$. Therefore, formula \eqref{4.1.2} is proved.
\end{proof}

\begin{lem}\label{useful lemma 2}
Suppose that $l_{\mu}\geq 1$. Then for each $i=1\lc l_{\mu}$, $i<j\leq n$, we have 
\begin{equation}\label{4.2.1}
    r_{0}(m_{i})=c_{1}\tilde{r}_{0}(m_{i})+M_{1},
\end{equation}
\begin{equation}\label{4.2.2}
    \det \bar{A}_{i,j}= c_{2} \det \bar{M}_{i,j}+M_{2},
\end{equation}
where $c_{1}$, $c_{2}$ are nonzero complex numbers, and $M_{1}$, $M_{2}$ belong to the ideal in $\mathcal{A}_{n}$ generated by $\tilde{b}_{q,-q+s}$, $s\in \Z_{\geq 0}$, $q>l+s$.
\end{lem}
\begin{proof}
Since $l_{\mu}\geq 1$, by condition \eqref{condition on bipartitions}, we have $l_{\lambda}=0$ and $l=l_{\mu}$.
Since $l\leq n$, one can use formula \eqref{4.2} for $k=0$ to see that 
\[ r_{0}(m_{i}) = (m_{i}-l)^{\underline{n-l}}\tilde{r}_{0}(m_{i})+M_{1},\]
where $M_{1}$ belongs to the ideal in $\mathcal{A}_{n}$ generated by $\tilde{b}_{q,-q+s}$, $s\in \Z_{\geq 0}$, $q>l+s$.
For $i\leq l$, we have  $m_{i}-l = -\mu_{i}+i-1-l<0$, therefore $(m_{i}-l)^{\underline{n-l}}\neq 0$, and formula \eqref{4.2.1} is proved.

To prove formula \eqref{4.2.2}, put $k=m_{j}-m_{i}$ and fix some $d=0\lc k-1$. First, suppose that $k-d+l\leq n$. Using \eqref{4.2}, one can see that up to summation with linear combinations of the elements $\tilde{b}_{q,-q+s}$, $s\in \Z_{\geq 0}$, $q>l+s$, the $(d+1)$-th row of $A_{i,j}$ is the $(d+1)$-th row of $M_{i,j}$ multiplied by $(m_{i}+d-l)^{\underline{n-l-k+d}}$. Notice that if $(d+1)$-th row of $M_{ij}$ is not deleted when we obtain $\bar{M}_{ij}$, then $m_{i}+d < m_{l} = -\mu_{l}+l-1<l$. Therefore, $m_{i}+d-l<0$, so  $(m_{i}+d-l)^{\underline{n-l-k+d}}\neq 0$. 

If $k-d+l>n$, then, using \eqref{4.3}, one can see that up to summation with linear combinations of the elements $\tilde{b}_{q,-q+s}$, $s\in \Z_{\geq 0}$, $q>l+s$, the $(d+1)$-th row of $M_{i,j}$ is the $(d+1)$-th row of $A_{i,j}$ multiplied by $(m_{i}-n+k)^{\underline{k-d+l-n}}$. Since $l_{\lambda}=0$, we have $m_{i}-n+k=m_{j}-n=-\mu_{j}+j-1-n< 0$, and $(m_{i}-n+k)^{\underline{k-d+l-n}}\neq 0$. Therefore, formula \eqref{4.2.2} is proved.
\end{proof}

The elements $\tilde{r}_{k}(\alpha)$ and $\bar{M}_{ij}$ are constructed using the Laurent series of the coefficients of the differential operator $D^{\C [y]}$ (see \eqref{Dcy}) at $z_{1}$. Therefore, we denote them now $\tilde{r}^{(1)}_{k}(\alpha)$ and $\bar{M}^{(1)}_{ij}$, respectively. Repeating this construction for the Laurent series at $z_{2}\lc z_{m}$, we get similar elements $\tilde{r}^{(a)}_{k}(\alpha)$ and $\bar{M}^{(a)}_{ij}$, $a=2\lc m$.

Define $a_{ik}^{(\infty)}\in\mathcal{A}_{n}$, $i=1\lc n$, $k\in\Z_{\geq 1}$ to be the coefficients of the Taylor series of the rational function $\sum_{a=1}^{m}\sum_{j=1}^{i}\frac{y_{ij}^{(a)}}{(u-z_{a})^{j}}$ at infinity, that is
\begin{equation}\label{a infty}
\sum_{a=1}^{m}\sum_{j=1}^{i}\frac{y_{ij}^{(a)}x^{j}}{\left(1-xz_{a}\right)^{j}}=\sum_{k=1}^{\infty}a_{ik}^{(\infty)}x^{k}. 
\end{equation}

Recall that $I_{1}= \{a\,\vert\,\mu^{(a)}=\emptyset\}$ and $I_{2}= \{a\,\vert\,\lambda^{(a)}=\emptyset\}$.

\begin{prop}\label{prop 4.3}
 The ideal $J_{n}(\bar{\nu})$ is generated by the following elements:
 \begin{equation}\label{gen2.1}
     \tilde{b}_{i,-i+k}^{(a)}, \quad 0\leq k< n-l^{(a)},\quad k+l^{(a)}<i\leq n,\quad 1\leq a\leq m,
 \end{equation}
 \begin{equation}\label{gen2.2}
     \tilde{r}^{(a)}_{0}(m_{n+1-j}^{(a)}),\quad\det \bar{M}^{(a)}_{n+1-i,n+1-j},\quad 1\leq j\leq l^{(a)},\quad j<i\leq n,\quad a\in I_{1},
 \end{equation}
 \begin{equation}\label{gen2.3}
     \tilde{r}^{(a)}_{0}(m_{i}^{(a)}),\quad\det \bar{M}^{(a)}_{i,j},\quad 1\leq i\leq l^{(a)},\quad i<j\leq n,\quad a\in I_{2},
 \end{equation}
 \begin{equation}\label{gen2.4}
     a_{ij}^{(\infty)}, \quad 1\leq i\leq n,\quad 1\leq j\leq i-1.
 \end{equation}
 \end{prop}
 \begin{proof}
 For each $i=1\lc n-l$ and $k\in\Z\geq 0$ we have 
 \begin{equation}\label{r tilde b}
     r_{k}(m_{l_{\mu}+i})=r_{k}(l_{\mu}+i-1) = \sum_{j=n-i+1}^{n}\tilde{b}_{j,-j+k}(i-1)^{\underline{n-j}}.
 \end{equation}
 Using this, one can check that 
 \begin{equation}\label{tilde b r}
     \tilde{b}_{n-i+1, -n+i-1+k} = \frac{1}{(i-1)!}(T^{i-1}r_{k})(l_{\mu}+i-1),
 \end{equation}
 where $Tf(\alpha)=f(\alpha)-f(\alpha -1)$ for every function $f(\alpha)$.
 
 Formulas \eqref{r tilde b} and \eqref{tilde b r} for $k=0$ show that the elements $\tilde{b}_{i,-i}$, $l<i\leq n$ are linear combinations of $r_{0}(m_{i})$, $l_{\mu}< i\leq n-l_{\lambda}$, and the other way around.
 
 Notice that for $l_{\mu}< i<j\leq n-l_{\lambda}$, $\bar{A}_{ij}$ is a $1\times 1$ matrix with the entry $r_{j-i}(m_{i})$. Therefore, for each $k=1\lc n-l-1$, formulas \eqref{r tilde b} and \eqref{tilde b r} show that $b_{i,-i+k}$, $k+l< i\leq n$ are linear combinations of $A_{i,i+k}$, $l_{\mu}< i\leq n-l_{\lambda}-k$, and the other way around. 
 
 Also, one can check that 
 \[ b^{(\infty)}_{i,-2i+j}=\sum_{l=0}^{\min (i, j-1)}c_{l}a^{(\infty)}_{i-l,j-l} \]
 for some complex numbers $c_{l}$ such that $c_{0}=1$, which implies that $b_{ij}^{(\infty)}$, $i=1\lc n$, $j=-2i+1\lc -i-1$ are linear combinations of $a_{kl}^{(\infty)}$, $k=1\lc n$, $l=1\lc k-1$, and the other way around. 
 
 Now, the proposition follows from Lemmas \ref{useful lemma 1} and \ref{useful lemma 2}
 \end{proof}
 
 \begin{rem} Notice that if $a\in I_{1}$ and $j>0$, then $\tilde{b}_{i, -j}^{(a)}=b_{i, -j}^{(a)}=y_{ij}^{(a)}$, in particular, the fact that the elements \eqref{gen2.1} belong to $J_{n}(\bar{\nu})$ implies that the coefficients of the differential operator $\cdet\bigl( \partial_{u} - \mathcal{L}(u)\bigr)$ (see \eqref{cdet expansion}) have a pole at $z^{(a)}$ of order not higher than $l^{(a)}$. 
 \end{rem}
 
 Denote the generating set of $J_{n}(\nu)$ from Proposition \ref{prop 4.3} as $S_{n}$. Let $pr_{\infty , n}:\,\mathcal{A}_{\infty}\rightarrow \mathcal{A}_{n}$ be the projection sending $y_{ij}^{(a)}$, $i>n$ to zero. For each $t\in\C$, denote by $\pi_{t}$ the projection $\mathcal{A}_{\infty}[w]\rightarrow \mathcal{A}_{\infty}[w]/(w-t)\cong \mathcal{A}_{\infty}$. The set $S_{n}$ is constructed in such a way that for every $a\in S_{n}$, there exists a unique element $a_{[w]}\in\mathcal{A}_{\infty}[w]$ such that $pr_{\infty , n}\circ\pi_{n}(a_{[w]})=a$, and for all $n'\geq n$, we have $pr_{\infty , n'}\circ\pi_{n'}(a_{[w]})\in S_{n'}$. Denote $K=\max_{a} l_{a}$, $S_{n,[w]}=\{a_{[w]}\vert a\in S_{n}\}$, $S_{[w]}=\cup_{n=K}^{\infty}S_{n, [w]}$, and let $J_{[w]}$ be the ideal in $\mathcal{A}_{\infty}[w]$ generated by $S_{[w]}$. In Theorem \ref{main} below, we will prove that for all but finitely many $t$, the ideal $\underline{J}_{t}=\pi_{t}(J_{[w]})$ is the kernel of the map $\underline{\phi}_{t}$. 
 
 \begin{rem}\label{R2}
Notice that without condition \eqref{condition on bipartitions}, the set $S_{n}$ should also contain elements $\bar{M}^{(a)}_{i,n+1-j}$ with $a$ such that $l^{(a)}_{\lambda}>0$ and $l^{(a)}_{\mu}>0$, $1\leq i\leq l^{(a)}_{\mu}$, $1\leq j\leq l^{(a)}_{\lambda}$. It is not clear if for such elements, there are interpolations $(\bar{M}^{(a)}_{i,n+1-j})_{[w]}\in\mathcal{A}_{\infty}[w]$.
 \end{rem}
 
For any $n\in\Z_{>0}$, let $J_{>n}\subset\mathcal{A}_{\infty}$ be the ideal generated by $y_{ij}^{a}$, $i>n$. The following lemma is the main technical tool for proving Theorem \ref{main}.
\begin{lem}\label{tool}
\begin{enumerate}
\item We have $\mathcal{A}_{\infty}[w]= \mathcal{A}_{K m}[w]+J_{[w]}$.
\item For all $n\geq K m$, we have $J_{>n}\subset \underline{J}_{n}$.
\end{enumerate}
\end{lem}
\begin{proof}
 Fix $i>1$. Notice that for each $k\in\Z_{\geq 1}$, the element $a_{ik}^{(\infty)}$ does not depend on $n$, therefore $a_{ik}^{(\infty)}=(a_{ik}^{(\infty)})_{[w]}\in S_{[w]}$. Denote by $\boldsymbol{a}_{i}^{(\infty)}$ the vector 
\[ \left( a_{i,1}^{(\infty)}\lc a_{i,i-1}^{(\infty)}\right)^{T}.\]
Denote by $\boldsymbol{y}_{i}$ the vector 
\[ \left( y_{i,1}^{(1)}, y_{i,1}^{(2)}\lc y_{i,1}^{(m)}, y_{i,2}^{(1)}\lc y_{i,2}^{(m)}\lc  y_{i,i-1}^{(1)}\lc  y_{i,i-1}^{(m)}\right)^{T}.\]
Using \eqref{a infty}, one can check that 
\[\boldsymbol{a}_{i}^{(\infty)}= L \boldsymbol{y}_{i},\]
where $L$ is an $(i-1)\times m(i-1)$ matrix defined as follows:
for any $j=0\lc i-2$, the elements from $(jm+1)$-th to $(j+1)m$-th of the $k$-th row are
\[ z_{1}^{k-1-j}\binom{k-1}{j},\, z_{2}^{k-1-j}\binom{k-1}{j}\lc z_{m}^{k-1-j}\binom{k-1}{j}.\]

The example with $m=2$, $i=5$ is given below:
\[ L=
\begin{pmatrix}
1 & 1 & 0 & 0 & 0 & 0 & 0 & 0 \\
z_{1} & z_{2} & 1 & 1 & 0 & 0 & 0 & 0 \\
z_{1}^{2} & z_{2}^{2} & 2z_{1} & 2z_{2} & 1 & 1 & 0 & 0 \\
z_{1}^{3} & z_{2}^{3} & 3z_{1}^{2} & 3z_{2}^{2} & 3z_{1} & 3z_{2} & 1 & 1
\end{pmatrix}.
\]

Let $L_{1}$ and $L_{2}$ be the blocks of the matrix $L$ such that $L_{1}$ is a square matrix, and $L=(L_{1}\vert L_{2})$. Let $\boldsymbol{y}'_{i}$ and $\boldsymbol{y}''_{i}$ be the vectors such that $\boldsymbol{y}'_{i}$ has $i-1$ elements, and $\boldsymbol{y}_{i}^{T} = ({\boldsymbol{y}'}_{i}^{T}\vert  {\boldsymbol{y}''}_{i}^{T})$. Then 
\[L_{1}\boldsymbol{y}'_{i}=\boldsymbol{a}^{(\infty)}_{i}-L_{2}\boldsymbol{y}''_{i}.\]

We will show that 
\begin{enumerate}
    \item the matrix $L_{1}$ is invertible,
    \item if $i>K m$, then any element of $\boldsymbol{y}''_{i}$ is a linear combination of an element from $J_{[w]}$ and elements of vectors $\boldsymbol{y}_{1}\lc \boldsymbol{y}_{i-1}$,
    \item if $n\geq K m$ and $i>n$, then any element of $\boldsymbol{y}''_{i}$ is a linear combination of an element from $\underline{J}_{n}$ and elements of vectors $\boldsymbol{y}_{n+1},\boldsymbol{y}_{n+2} \lc \boldsymbol{y}_{i-1}$.
\end{enumerate}
Then the lemma will follow by induction on $i$.

If $i-1>m$, then introduce additional parameters $z_{m+1}\lc z_{i-1}$. 
Let $V_{i-1}(z_{1}\lc z_{i-1})$ be the Vandermonde determinant $\det (z_{a}^{b-1})_{1\leq a,b\leq i-1}$. One can check that 
\begin{equation}\label{detL1} 
\det L_{1} = \prod_{l,a}\raisebox{1.5pt}{$\frac{\raisebox{2.5pt}{\scalebox{1.05}{\ensuremath 1}}}{\raisebox{-5pt}{$l!$}}$}
 \,
    \begin{tabular}{p{11pt}}
     $\partial^{l}$ \\
     \hline
     \raisebox{-3pt}{\shifttext{-7pt}{$\partial z^{l}_{lm+a}$}}
  \end{tabular} 
  \quad\,\, \bigg\rvert_{z_{lm+a}=z_{a}} V_{i-1}(z_{1}\lc z_{i-1}),
\end{equation}  
where the product is taken over $l$, $a$ such that $1\leq a\leq m$, $l\geq 1$, and $lm+a\leq i-1$. 

Using $V_{i-1}(z_{1}\lc z_{i-1})=\prod_{1\leq j< j'\leq i-1}(z_{j'}-z_{j})$, we see that the right hand side of \eqref{detL1} has the form $(-1)^{\varepsilon}\prod_{1\leq j< j'\leq m}(z_{j'}-z_{j})^{d_{j',j}}$ with some $d_{j',j}>0$ and $\varepsilon\in\{0,1\}$. Since $z_{a}\neq z_{b}$ for $a\neq b$, we have $\det L_{1}\neq 0$, therefore, $L_{1}$ is invertible.

One can check that for any $n>i$, $j\leq i$, $a=1\lc m$, we have 
\begin{equation}\label{b tilde to b}
\tilde{b}_{i,-j}^{(a)}=b_{i,-j}^{(a)}+\sum_{s=1}^{\min(l_{\mu},i)}b_{i-s,-j+s}^{(a)}(n-i+1)(n-i+2)\dots (n-i+s).
\end{equation}
Suppose that $i>Km$. Take $j\leq i$ such that $j>K$. Then for all $a=1\lc m$, $\tilde{b}_{i,-j}^{(a)}\in S_{n}$, and  $b_{i-s,-j+s}^{(a)}=y_{i-s,j-s}^{(a)}$, $s=0\lc l_{\mu}^{(a)}$. Therefore, formula \eqref{b tilde to b} gives the following expression for $(\tilde{b}_{i,-j}^{(a)})_{[w]}\in S_{[w]}$:
\begin{equation}\label{bw tilde}
(\tilde{b}_{i,-j}^{(a)})_{[w]}=y_{i,-j}^{(a)}+\sum_{s=1}^{\min(l_{\mu},i)}y_{i-s,-j+s}^{(a)}(w-i+1)(w-i+2)\dots (w-i+s).
\end{equation}
Moreover, for any integer $n\in\Z$, $1\leq n <i$, if we put $w=n$ in \eqref{bw tilde}, then all terms in the sum on the right hand side of \eqref{bw tilde} corresponding to $i-s\leq n$ will become zero. 

Notice that if $i>Km$, then $\boldsymbol{y}''_{i}$ contains only elements $y_{ij}^{(a)}$ with $j>K$, so we proved the lemma.

\end{proof}

 Denote $\mathcal{F}_{[w]}=\mathcal{A}_{\infty}[w]/J_{[w]}$. By Lemma \ref{tool}, part (1), $\mathcal{F}_{[w]}$ is a finitely generated $\C [w]$-algebra 
 
 Consider the torsion ideal of $\mathcal{F}_{[w]}$
 \[ \operatorname{Tors}\,\mathcal{F}_{[w]} = \{ a\in \mathcal{F}_{[w]}\,\vert\, pa=0\text{ for some } p\in\C [w]\},\]
 and the quotient $\mathcal{F}^{\,t.f.}_{[w]}=\mathcal{F}_{[w]}/\operatorname{Tors}\,\mathcal{F}_{[w]}$. Clearly,  $\mathcal{F}^{\,t.f.}_{[w]}$ is a finitely generated $\C [w]$-algebra as well.
 
 For a commutative ring $R$ and its elements $r_{1}\lc r_{d}\in R$, denote by $R_{r_{1}\lc r_{d}}$ the localization of $R$ by the multiplicative set $\{r_{1}^{i_{1}}\dots r_{d}^{i_{d}}\,\vert\, i_{1}\lc i_{d}\in\Z_{\geq 0}\}.$ Notice that the kernel of the homomorphism $R\rightarrow R_{r_{1}\lc r_{d}}$, $r\mapsto \frac{r}{1}$ is exactly the set of elements $r\in R$ such that $r_{i}r=0$ for some $i=1\lc d$.
 
 Let us use the same notation for a polynomial in $\C [w]$ and its image in $\mathcal{F}_{[w]}$, $\mathcal{F}^{\, t.f.}_{[w]}$, or $\mathcal{B}_{[w]}(\bar{\lambda},\bar{\mu})$. Since there are no elements $f\in \mathcal{F}^{\, t.f.}_{[w]}$ such that $pf=0$ for some $p\in\C [w]$, we can identify $\mathcal{F}^{\, t.f.}_{[w]}$ with the corresponding subalgebra of $(\mathcal{F}^{\, t.f.}_{[w]})_{p_{1}\lc p_{d}}$ for any $p_{1}\lc p_{d}\in\C [w]$.
 
 \begin{lem}\label{loc}
 There exist finitely many nonzero polynomials $p_{1}\lc p_{d}\in\C [w]$ such that $(\mathcal{F}^{\, t.f.}_{[w]})_{p_{1}\lc p_{d}}$ is a free $(\C [w])_{p_{1}\lc p_{d}}$-module.
 \end{lem}
 \begin{rem*}
This lemma holds true if we replace $\mathcal{F}^{\, t.f.}_{[w]}$ with any finitely generated commutative $\C[w]$-algebra, therefore, it is a consequence of a general fact in algebra. Since we did not find the exact statement in the literature, we provide a proof here. 
 \end{rem*}
 \begin{proof}
 Let $x_{1}\lc x_{k}$ be the generators of $\mathcal{F}^{\, t.f.}_{[w]}$ as a $\C [w]$-algebra. For any $\vec{i}=(i_{1}\lc i_{k})\in\Z_{\geq 0}^{k}$, denote $x^{\vec{i}}=x_{1}^{i_{1}}\dots x_{k}^{i_{k}}$. Consider the "inverse lexicographic" ordering on $\Z_{\geq 0}^{k}$, that is
 $(i_{1}\lc i_{k})<(j_{1}\lc j_{k})$
 if and only if for the maximal index $k'$ such that $i_{k'}\neq j_{k'}$, we have $i_{k'}< j_{k'}$.  If $\vec{i},\vec{j}\in\Z_{\geq0}^{k}$ are such that $\vec{i}<\vec{j}$ we will say that $x^{\vec{i}}$ is smaller then $x^{\vec{j}}$. Let $\mathcal{Z}$ be the subset of $\mathcal{F}^{\, t.f.}_{[w]}$ consisting of monomials $x^{\vec{i}}$ such that $p_{\vec{i}}x^{\vec{i}}+\sum_{\vec{j}<\vec{i}}q_{\vec{j}}x^{\vec{j}}=0$ for some $p_{\vec{i}},q_{\vec{j}}\in \C [w]$, $p_{\vec{i}}\neq 0$. Since $\mathcal{F}^{\,t.f.}_{[w]}$ is a finitely generated $\C[w]$-algebra, by Hilbert's basis theorem, the ideal $M_{\mathcal{Z}}\subset \mathcal{F}^{\, t.f.}_{[w]}$ generated by $\mathcal{Z}$ has a finite generating set $\widetilde{\mathcal{Z}}$.
 
 Notice that $\mathcal{Z}$ generates $M_{\mathcal{Z}}$ as a $\C [w]$-module. Therefore, we can assume that $\widetilde{\mathcal{Z}}$ is a finite subset of $\mathcal{Z}$. Then $\widetilde{\mathcal{Z}}=\{x^{\vec{i}_{1}}\lc x^{\vec{i}_{d}}\}$ for some $\vec{i}_{1}\lc \vec{i}_{d}\in \Z^{k}_{\geq 0}$, and there exist polynomials $p_{1}\lc p_{d}\in\C [w]$ such that $p_{j}x^{\vec{i}_{j}}+\sum_{\vec{s}<\vec{i}_{j}}q_{j,\vec{s}}x^{\vec{s}}=0$, $j=1\lc d$ for some $q_{j,\vec{s}}\in\C [w]$. We claim that $(\mathcal{F}^{\, t.f.}_{[w]})_{p_{1}\lc p_{d}}$ is a free $(\C [w])_{p_{1}\lc p_{d}}$-module with a basis $\bar{\mathcal{Z}}$ consisting of all monomials $x^{\vec{i}}$ that do not belong to $\mathcal{Z}$. 
 
 First, let us show that the set $\bar{\mathcal{Z}}$ is linearly independent. Suppose that for some $x^{\vec{j}_{1}}\lc x^{\vec{j}_{s}}\in\bar{\mathcal{Z}}$ and nonzero $q_{\vec{j}_{1}}\lc q_{\vec{j}_{s}}\in(\C [w])_{p_{1}\lc p_{d}}$, we have $\sum_{i=1}^{s}q_{\vec{j}_{i}}x^{\vec{j}_{i}}=0$. We can find nonzero polynomial $p\in\C [w]$ such that $pq_{\vec{j}_{i}}\in\C [w]$, $i=1\lc s$. Denote $\vec{j}=\max_{i}\vec{j}_{i}$. Then $\sum_{i=1}^{s}pq_{\vec{j}_{i}}x^{\vec{j}_{i}}=0$ implies that $x^{\vec{j}}\in \mathcal{Z}$, which is a contradiction since $\bar{\mathcal{Z}}\cap\mathcal{Z}=\emptyset$. 
 
 To prove that $\bar{\mathcal{Z}}$ generates $(\mathcal{F}^{\, t.f.}_{[w]})_{p_{1}\lc p_{d}}$ as a $(\C [w])_{p_{1}\lc p_{d}}$-module, let us first assume that $d=1$, that is $\widetilde{\mathcal{Z}}$ consists of a single element $x^{\vec{i}_{1}}$. 
 
 Recall that $p_{1}x^{\vec{i}_{1}}+\sum_{\vec{s}<\vec{i}_{1}}q_{1,\vec{s}}x^{\vec{s}}=0$ for some $q_{1,\vec{s}}\in\C [w]$. Therefore, in $(\mathcal{F}^{\, t.f.}_{[w]})_{p_{1}}$, we have 
 \begin{equation}\label{4.6.1}
     x^{\vec{i}_{1}} = -\sum_{\vec{s}<\vec{i}_{1}}\left(\frac{q_{1,\vec{s}}}{p_{1}}\right)x^{\vec{s}}.
 \end{equation}
 Consider a monomial $f=x^{\vec{i}_{1}+\vec{j}}\in\mathcal{Z}$ for some $\vec{j}=(j_{1}\lc j_{k})\in\Z_{\geq 0}^{k}$.  Let $k'$ be the minimal positive integer such that for all $k''=k'+1\lc k$, we have $j_{k''}=0$.
 We will prove by induction on $k'$ that $f$ is a linear (over $(\C [w])_{p_{1}}$) combination of monomials from $\bar{\mathcal{Z}}$ smaller then $f$. The base $k'=0$ is given by the formula \eqref{4.6.1}. If $k'>0$, then by induction assumption $f$ is a linear combination of monomials from $\bar{\mathcal{Z}}$ smaller then $f$ and monomials of the form $x^{\vec{i}_{1}}x_{1}^{j'_{1}}x_{2}^{j'_{2}}\dots x_{k'}^{j'_{k'}}$, where $j'_{k'}<j_{k'}$. Then, we proceed by induction on $j_{k'}$.
 
 If $d>1$, then we use similar arguments to prove that
 \begin{enumerate}
     \item monomials that do not belong to the ideal generated by $x^{\vec{i}_{1}}$ generate $(\mathcal{F}^{\, t.f.}_{[w]})_{p_{1}}$ as a $(\C [w])_{p_{1}}$-module,
     \item monomials that do not belong to the ideal generated by $x^{\vec{i}_{1}}$ and $x^{\vec{i}_{2}}$ generate $(\mathcal{F}^{\, t.f.}_{[w]})_{p_{1},p_{2}}$ as a $(\C [w])_{p_{1},p_{2}}$-module, and so on.
 \end{enumerate}
 \end{proof}
 
 Let $p_{1}\lc p_{d}$ be the polynomials from Lemma \ref{loc}. Denote $\mathcal{F}^{\, loc}_{[w]}=(\mathcal{F}^{\, t.f.}_{[w]})_{p_{1}\lc p_{d}}$. Recall the ideal $\underline{J}_{t}\subset \mathcal{A}_{\infty}$ introduced above. Denote $\underline{\mathcal{F}}_{t}=\mathcal{A}_{\infty}/\underline{J}_{t}\cong \mathcal{F}_{[w]}/(w-t)$. 
 \begin{lem}\label{loc evaluation}
 For all but finitely many $t\in\C$, we have
 \[ \mathcal{F}^{\, loc}_{[w]}/(w-t)\cong\underline{\mathcal{F}}_{t}.\]
 \end{lem}
\begin{proof}
Denote by $\pi_{t}$ the projection $\mathcal{F}_{[w]}\rightarrow \underline{\mathcal{F}}_{t}$. Notice that if $a\in\mathcal{F}_{[w]}$, $p\in\C [w]$, and  $t\in\C$ are such that $pa=0$, and $t$ is not a root of $p$, then $\pi_{t}(a)=0$. Indeed, let $t_{0}$ be a root of $p$. Then $p=(w-t_{0})\tilde{p}$ for some $\tilde{p}\in\C [w]$ such that $\deg\tilde{p}=\deg p -1$. We have
\begin{equation*}
    0  = \pi_{t}\left[\bigl((w-t)\tilde{p}-(w-t_{0})\tilde{p}\bigr)a\right] 
     = (t-t_{0})\pi_{t}(\tilde{p}a).
\end{equation*}
Since $t\neq t_{0}$, this implies $\pi_{t}(\tilde{p}a)=0$. If $t_{1}$ is a root of $\tilde{p}$, then a similar calculation gives $\pi_{t}(\dbtilde{p}a)=0$, where $\tilde{p}=(w-t_{1})\dbtilde{p}$, $\deg\dbtilde{p}=\deg \tilde{p} -1$. Repeating this argument $\deg p$ times, we get $\pi_{t}(a)=0$.

Since $\mathcal{F}_{[w]}$ is a finitely generated $\C [w]$-algebra, by Hilbert's basis theorem, the ideal $\operatorname{Tors}\,\mathcal{F}_{[w]}$ is finitely generated. Let $a_{1}\lc a_{k}$ be the generators of $\operatorname{Tors}\,\mathcal{F}_{[w]}$, and let $q_{1}\lc q_{k}\in\C [w]$ be such that $q_{i} a_{i}=0$, $i=1\lc k$. Then
\begin{equation}\label{Tors}
    \operatorname{Tors}\,\mathcal{F}_{[w]}\subset (w-t)
\end{equation}
for all $t$ except the roots of $q_{1}\lc q_{k}$. Therefore, $\mathcal{F}^{\, t.f}_{[w]}/(w-t)\cong\underline{\mathcal{F}}_{t}$ for all but finitely many $t\in\C$. Since $\mathcal{F}^{\, t.f}_{[w]}/(w-t)\cong \mathcal{F}^{\, loc}_{[w]}/(w-t)$ for all $t\in\C$ except the roots of the polynomials $p_{1}\lc p_{d}$, the lemma follows.
\end{proof}

\begin{lem}\label{4.8}
Let $\pi_{t}$ denote the projection $\mathcal{F}^{\, loc}_{[w]}\rightarrow\mathcal{F}^{\, loc}_{[w]}/(w-t)$. Suppose that for some $p\in \mathcal{F}^{\, loc}_{[w]}$ and $N_{0}\in\Z_{>0}$, we have $\pi_{n}(p)=0$, $n\in\Z_{\geq N_{0}}$. Then $p=0$. 
\end{lem}
\begin{proof}
Since by Lemma \ref{loc}, $\mathcal{F}^{\, loc}_{[w]}$ is a free $(\C [w])_{p_{1}\lc p_{d}}$-module, Lemma \ref{4.8} can be proved similarly to Lemma \ref{1.1}.
\end{proof}

Recall the algebra $\mathcal{B}_{[w]}(\bar{\lambda},\bar{\mu})$, see Section \ref{4}. Let $p_{1}\lc p_{d}\in\C [w]$ be the polynomials from Lemma \ref{loc}. Denote $\mathcal{B}^{loc}_{[w]}(\bar{\lambda},\bar{\mu})=\bigl(\mathcal{B}_{[w]}(\bar{\lambda},\bar{\mu})\bigr)_{p_{1}\lc p_{d}}$. Since $\mathcal{B}_{[w]}(\bar{\lambda},\bar{\mu})$ is a torsion-free $\C [w]$-module, we can identify it with the corresponding subalgebra of $\mathcal{B}^{loc}_{[w]}(\bar{\lambda},\bar{\mu})$. For all $t\in\C$ except the roots of the polynomials $p_{1}\lc p_{d}$, we have $\mathcal{B}^{loc}_{[w]}(\bar{\lambda},\bar{\mu})/(w-t)\cong\mathcal{B}_{[w]}(\bar{\lambda},\bar{\mu})/(w-t)$. Together with formula \eqref{B eval}, this implies 
\begin{equation}\label{B loc evaluation}
    \mathcal{B}^{loc}_{[w]}(\bar{\lambda},\bar{\mu})/(w-t)\cong \underline{\mathcal{B}}_{t}(\bar{\lambda},\bar{\mu})
\end{equation}
for all but finitely many $t\in\C$.

We are finally ready to prove the main result of the paper:
\begin{thm}\label{main}
For all but finitely many $t\in\C$, the kernel of the homomorphism 
\begin{equation*}
    \begin{split}
\underline{\phi}_{t}:\quad & \mathcal{A}_{\infty} \rightarrow \underline{\mathcal{B}}_{t}(\bar{\lambda},\bar{\mu}), \\
& y^{(a)}_{ij}\mapsto \underline{B}_{ij}^{(a)}(\bar{\lambda},\bar{\mu}).
    \end{split}
\end{equation*}
is the ideal $\underline{J}_{t}$.
\end{thm}
\begin{proof}
Let $B_{ij}^{(a)}[w]$ be the generators of $\mathcal{B}_{[w]}(\bar{\lambda},\bar{\mu})$, which are analogous to the generators $\underline{B}_{ij}^{(a)}(\bar{\lambda},\bar{\mu})$ of $\underline{\mathcal{B}}_{t}(\bar{\lambda},\bar{\mu})$. Consider the map $\phi_{[w]}:\, \mathcal{A}_{\infty} [w]\rightarrow \mathcal{B}_{[w]}(\bar{\lambda},\bar{\mu})$ sending $y_{ij}^{(a)}$ to $B_{ij}^{(a)}[w]$ and $w$ to $w$. 

Let us use the notation $\pi_{t}$ for the projection $A\rightarrow A/(w-t)$, where $A$ is one of the algebras $\mathcal{A}_{\infty} [w]$, $\mathcal{F}_{[w]}$, $\mathcal{F}^{\,loc}_{[w]}$, $\mathcal{B}_{[w]}(\bar{\lambda},\bar{\mu})$, $\mathcal{B}^{loc}_{[w]}(\bar{\lambda},\bar{\mu})$. For all $t\in\C$, we have 
\begin{equation*}
\pi_{t}\circ \phi_{[w]} = \underline{\phi}_{t}\circ\pi_{t}.
\end{equation*}

Let $a$ be an element of $S_{[w]}$. Since by definition, $S_{[w]}=\cup_{n=K}^{\infty}S_{n,[w]}$, we have $a\in S_{N'(a),[w]}$ for some $N'(a)\in\Z_{\geq K}$. Then $pr_{\infty, n}\circ\pi_{n}(a)\in S_{n}$ for all $n\geq N'(a)$. On the other hand, since $\mathcal{A}_{\infty} [w] = \cup_{n=1}^{\infty}\mathcal{A}_{n} [w]$ and $\mathcal{A}_{n_{1}} [w]\subset \mathcal{A}_{n_{2}} [w]$ for $n_{1}<n_{2}$, there exists $N''(a)\in \Z_{>0}$ such that $a\in \mathcal{A}_{n} [w]$ for all $n\geq N''(a)$. Therefore, if $n\geq N'(a)$ and $n\geq N''(a)$, then $\pi_{n}(a)\in S_{n}$. 

Recall that $N=\sum_{a}(|\lambda^{(a)}|+|\mu^{(a)}|)$. Set $N_{0}=\max (N'(a),\, N''(a), N)$. Then for all $n\geq N_{0}$, we have
\[\pi_{n}\circ\phi_{[w]}(a) = \underline{\phi}_{n}\circ\pi_{n}(a)=\phi_{n}\circ\pi_{n}(a) = 0.\]
Here, in the second equality, we used that 
\begin{enumerate}
\item since $n\geq N''(a)$, we have $\pi_{n}(a)\in\mathcal{A}_{n}$,
\item since $n\geq N$, by Lemma \ref{B is B}, we have $\underline{\mathcal{B}}_{n}(\bar{\lambda},\bar{\mu})\cong\mathcal{B}_{n}(\bar{\nu})$, therefore, $\underline{\phi}_{n}\big\rvert_{\mathcal{A}_{n}}=\phi_{n}$.
\end{enumerate}
In the third equality, we used that since $n\geq N'(a)$ and $n\geq N''(a)$, we have $\pi_{n}(a)\in S_{n}\subset\ker \phi_{n}$. 

Since $\pi_{n}\circ\phi_{[w]}(a)=0$ for all $n\geq N_{0}$, by Lemma \ref{1.1}, we have $\phi_{[w]}(a)=0$. Thus, $J_{[w]}\subset\ker\phi_{[w]}$, in particular, $\phi_{[w]}$ induces a homomorphism $\phi_{[w]}^{\mathcal{F}}:\mathcal{F}_{[w]}\rightarrow \mathcal{B}_{[w]}(\bar{\lambda},\bar{\mu})$.
The map $\phi_{[w]}^{\mathcal{F}}\circ\pi_{t}$ factors through the map $\pi_{t}:\mathcal{F}_{[w]}\rightarrow \underline{\mathcal{F}}_{t}$. Let $\underline{\phi}_{t}^{\mathcal{F}}:\underline{\mathcal{F}}_{t}\rightarrow \underline{\mathcal{B}}_{t}(\bar{\lambda},\bar{\mu})$ be the homomorphism such that $\pi_{t}\circ\phi_{[w]}^{\mathcal{F}} = \underline{\phi}_{t}^{\mathcal{F}}\circ\pi_{t}$.

Observe that $\operatorname{Tors}\,\mathcal{F}_{[w]}\subset \ker \phi_{[w]}^{\mathcal{F}}$. Indeed, if $a\in \operatorname{Tors}\,\mathcal{F}_{[w]}$, then by \eqref{Tors}, $\pi_{t}(a)=0$ for all $t\in\C\setminus I$, where $I$ is a finite set. Let $N_{0}\in\Z_{>0}$ be greater then any integer in $I$. Then for any $n\geq N_{0}$, we have
\[\pi_{n}\circ\phi_{[w]}^{\mathcal{F}}(a) = \underline{\phi}_{n}^{\mathcal{F}}\circ\pi_{n}(a) = 0. \]
Therefore, by Lemma \ref{1.1}, $\phi_{[w]}^{\mathcal{F}}(a) =0$.

Since $\operatorname{Tors}\,\mathcal{F}_{[w]}\subset \ker \phi_{[w]}^{\mathcal{F}}$, the map $\phi_{[w]}^{\mathcal{F}}$ induces a homomorphism $(\phi_{[w]}^{\mathcal{F}})^{t.f.}:\mathcal{F}^{\,t.f.}_{[w]}\rightarrow \mathcal{B}_{[w]}(\bar{\lambda},\bar{\mu})$, which has a unique extension $(\phi_{[w]}^{\mathcal{F}})^{\,loc}:\mathcal{F}^{loc}_{[w]}\rightarrow \mathcal{B}^{loc}_{[w]}(\bar{\lambda},\bar{\mu})$. We are going to prove that $(\phi_{[w]}^{\mathcal{F}})^{\,loc}$ is an isomorphism. Then the theorem will follow from Lemma \ref{loc evaluation} and formula \eqref{B loc evaluation}.

First, observe that if $n\geq Km$ and $n\geq N$, then $\ker \underline{\phi}_{n}^{\mathcal{F}} = 0$. Indeed, for such $n$, we can write
\begin{align*}
    \ker \underline{\phi}_{n} & = J_{n}(\bar{\nu})+I_{>n} = pr_{\infty, n}\, \underline{J}_{n}+I_{>n} \\
    & = \underline{J}_{n}+I_{>n}=\underline{J}_{n}.
\end{align*}
Here, in the first equality, we used that since $n\geq N$, by Lemma \ref{B is B}, we have $\underline{\mathcal{B}}_{n}(\bar{\lambda},\bar{\mu})\cong\mathcal{B}_{n}(\bar{\lambda},\bar{\mu})$; in the last equality, we used that $n\geq Km$, and applied Lemma \ref{tool}, part (2). We obtained that $\ker \underline{\phi}_{n}=\underline{J}_{n}$, therefore, $\ker\,\underline{\phi}_{n}^{\mathcal{F}} = 0$. 

Now, let $\widetilde{I}$ be the minimal finite subset of $\C$, such that \eqref{B loc evaluation} holds for all $t\in\C\setminus\widetilde{I}$, and let $N_{\widetilde{I}}\in\Z_{>0}$ be greater then any integer in $\widetilde{I}$. Consider an element $a$ from $\ker\, (\phi_{[w]}^{\mathcal{F}})^{\,loc}$, and set $N_{0}=\max (Km, N_{\widetilde{I}}, N)$. Take any $n\in\Z_{\geq N_{0}}$. Since $n\geq N_{\widetilde{I}}$, we have $\pi_{n}\circ(\phi_{[w]}^{\mathcal{F}})^{\,loc} = \underline{\phi}_{n}^{\mathcal{F}}\circ\pi_{n}$. Therefore, $\pi_{n}(a)\in\ker\,\underline{\phi}_{n}^{\mathcal{F}}$. Since $n\geq Km$ and $n\geq N$, we have $\ker\,\underline{\phi}_{n}^{\mathcal{F}} = 0$, so $\pi_{n}(a)=0$. By Lemma \ref{4.8}, we have $a=0$, therefore, $(\phi_{[w]}^{\mathcal{F}})^{\,loc}$ is an isomorphism.

\end{proof}

\section{On ratios of differential operators with no monodromy}\label{7}
\subsection{Pseudo-differential operators}
Let $\mathcal{A}$ be an associative unital commutative algebra, and let $\partial:\, \mathcal{A}\rightarrow\mathcal{A}$ be a derivation on $\mathcal{A}$. Then a pseudo-differential operator of order $t\in\C$ with coefficitents in $\mathcal{A}$ is a formal series of the form
\begin{equation}\label{ps dif op}
\sum_{r=0}^{\infty}a_{r}\partial^{t-r},\quad\quad a_{r}\in\mathcal{A},\quad a_{0}\neq 0.
\end{equation}
Denote the set of all pseudo-differential operators with coefficients in $\mathcal{A}$ as $\Psi\mathcal{D}(\mathcal{A})$. One can add and multiply series of the form \eqref{ps dif op} in a usual way if for all $a\in\mathcal{A}$, we set
\[ \partial^{\mu}a=\sum_{i=0}^{\infty}\frac{\mu^{\underline{i}}}{i!}(\partial^{i}a)\partial^{\mu-i}. \]
It can be checked that this multiplication is well-defined and associative, thus, making $\Psi\mathcal{D}(\mathcal{A})$ an associative algebra.
\subsection{}\label{7.2}
Fix a sequence of $m$ pairwise different complex numbers $\bar{z}=(z_{1}\lc z_{m})$. Recall the differential operator $D^{\C [y]}\in\mathcal{D}(\mathcal{A}_{n})$ defined in \eqref{Dcy}. We wil now consider two algebras, $\mathcal{A}_{n}$ and $\mathcal{A}_{n'}$ for some other natural number $n'$, together. Therefore, to distinguish between the differential operator $D^{\C [y]}$ associated with $\mathcal{A}_{n}$ and the differential operator $D^{\C [y]}$ associated with $\mathcal{A}_{n'}$, we will denote them as $D_{n}$ and $D_{n'}$, respectively. Both $D_{n}$ and $D_{n'}$ can be thought of as elements of $\Psi\mathcal{D}\bigl((\mathcal{A}_{n}\otimes\mathcal{A}_{n'})(u)\bigr)$. Then the pseudo-differential operator $D_{n|n'}=D_{n}D_{n'}^{-1}\in\Psi\mathcal{D}\bigl((\mathcal{A}_{n}\otimes\mathcal{A}_{n'})(u)\bigr)$ is of the form
\[D_{n|n'}=\partial_{u}^{n-n'}+\sum_{i=1}^{\infty}\sum_{a=1}^{m}\sum_{j=1}^{i}\frac{c_{ij}^{(a)}}{(u-z_{a})^{j}}\partial_{u}^{n-n'-i} .\]
for some $c_{ij}^{(a)}\in\mathcal{A}_{n}\otimes\mathcal{A}_{n'}$. Let $\mathcal{A}_{n|n'}$ be the unital subalgebra of $\mathcal{A}_{n}\otimes\mathcal{A}_{n'}$ generated by $c_{ij}^{(a)}$, $i\in \Z_{\geq 1}$, $j=1\lc i$, $a=1\lc m$. 

Let $\bar{\nu}=(\nu^{(1)}\lc \nu^{(m)})$ be a sequence of dominant integral $\gl_{n}$-weights, where $\nu^{(i)}=(\nu^{(i)}_{1}\lc \nu^{(i)}_{n})$ such that $\nu^{(i)}_{j}\geq 0$, $i=1\lc m$, $j=1\lc n$. Let $\bar{\eta}=(\eta^{(1)}\lc \eta^{(m)})$ be a sequence of dominant integral $\gl_{n'}$-weights, where $\eta^{(i)}=(-\eta^{(i)}_{n'}\lc -\eta^{(i)}_{1})$ such that $\eta^{(i)}_{j}\geq 0$, $i=1\lc m$, $j=1\lc n'$. For each $i=1\lc m$, $j=1\lc n'$, let ${}^{c}\eta^{(i)}_{j}$ denote the greatest number such that $\eta^{(i)}_{j'}\geq j$ for each $j'\leq {}^{c}\eta^{(i)}_{j}$. 
For each $i=1\lc m$, $j\in Z_{\geq 1}$, define a non-negative integer $\lambda^{(i)}_{j}$ as follows
\[ \lambda^{(i)}_{j}=
\begin{cases}
\nu_{j}^{(i)} & j=1\lc n, \\
{}^{c}\eta^{(i)}_{j-n} & j=n\lc n+n', \\
0 & j>n+n'.
\end{cases}
\]
We will assume that $\bar{\nu}$ and $\bar{\eta}$ are such that for each $i=1\lc k$, the sequence $\lambda^{(i)}=(\lambda^{(i)}_{1}, \lambda^{(i)}_{2}, \dots)$ is a partition satisfying the property $\lambda^{(i)}_{n+1}\leq n'$ (such partition is called a $n|n'$-hook partition). We will write $\bar{\lambda}$ for the sequence of partitions $(\lambda^{(1)}\lc\lambda^{(k)})$.

It is convenient to depict a partition $(\lambda_{1},\lambda_{2}, \dots)$ by its Young diagram, that is, by a collection of rows of boxes aligned by the left-most box such that the $j$-th row consists of $\lambda_{j}$ boxes. (Here, we number the rows from top to bottom). Then, for some fixed $i$, the Young diagram of the partition $\lambda^{(i)}$ defined above looks as follows:
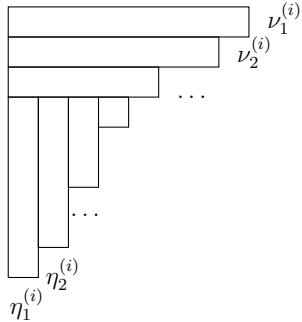
\begin{center}
\begin{tikzpicture}
\draw[black] (0,0) rectangle ++(3.2,-0.4)  (3.3,-0.15) node[anchor=west]{\scalebox{0.8}{$\nu^{(i)}_{1}$}};
\draw[black] (0,-0.4) rectangle ++(2.8,-0.4) (2.9,-0.6) node[anchor=west]{\scalebox{0.8}{$\nu^{(i)}_{2}$}};
\draw[black] (0,-0.8) rectangle ++(2,-0.4) (2.1,-1.2) node[anchor=west]{\scalebox{0.8}{$\dots$}};
\draw[black] (0,-1.2) rectangle ++(0.4,-2.4) (0.25,-3.6) node[anchor=north]{\scalebox{0.8}{$\eta^{(i)}_{1}$}};
\draw[black] (0.4,-1.2) rectangle ++(0.4,-2) (0.73,-3.2) node[anchor=north]{\scalebox{0.8}{$\eta^{(i)}_{2}$}};
\draw[black] (0.8,-1.2) rectangle ++(0.4,-1.2) (1.05,-2.6) node[anchor=north]{\scalebox{0.8}{$\dots$}};
\draw[black] (1.2,-1.2) rectangle ++(0.4,-0.4);
\end{tikzpicture}
\end{center}

Recall the ideals $J_{n}(\bar{\nu})$ and $\underline{J}_{t}(\bar{\lambda},\bar{\mu})$ introduced in Sections \ref{5.2} and \ref{6}, respectively (here, we indicated the dependence of $\underline{J}_{t}$ on $\bar{\lambda}$ and $\bar{\mu}$). Let $\bar{\emptyset}$ denote the sequence of partitions with all elements equal to $(0,0,\dots )$.
Introduce two ideals, $J_{n|n'}$ and $\tilde{J}_{n|n'}$, of the algebra $\mathcal{A}_{n|n'}$ as follows:
\[J_{n|n'}=(J_{n}(\bar{\nu})\otimes J_{n'}(\bar{\eta}))\cap\mathcal{A}_{n|n'},\]
and $\tilde{J}_{n|n'}$ is the image of the ideal $\underline{J}_{t}(\bar{\lambda},\bar{\emptyset})$ under the surjective homomorphism 
\begin{align*}
    \mathcal{A}_{\infty} &\rightarrow  \mathcal{A}_{n|n'},\\
    y_{ij}^{(a)} & \mapsto  c_{ij}^{(a)},
\end{align*}
where $\bar{\nu}$, $\bar{\eta}$, $\bar{\lambda}$ are related as described above. The following theorem is the main result of this section:

\begin{thm}\label{main2}
    We have $\tilde{J}_{n|n'}\subset J_{n|n'}$
\end{thm}

\subsection{} 
To prove Theorem \ref{main2}, we will first need to prove a few lemmas. Again, let $\mathcal{A}$ be an arbitrary associative commutative unital algebra. For this section, let $\partial$ denote the derivation $\partial_{x}$ on $\mathcal{A}((x))$.

\begin{lem}\label{l1}
let $D$ be the differential operator given by the formula \eqref{D}. Suppose that the differential equation $Df=0$ has solutions $f_{!}\lc f_{n}\in\mathcal{A}((x))$ of the form \eqref{solution} with some fixed integers $m_{1}<m_{2}<\dots m_{n}$. Denote $\nu_{i}=m_{n+1-i}-n+i$. Then $D=D_{1}D_{2}\dots D_{n}$, where 
\[ D_{i}=\partial-\frac{\nu_{i}}{x}+d_{0}^{(i)}+d_{1}^{(i)}x+d_{2}^{(i)}x^{2}+\dots , \quad i=1\lc n,\]
for some $d^{(i)}_{j}\in\mathcal{A}$.
\end{lem}
\begin{proof}
For any $h_{1}\lc h_{k}\in\mathcal{A}((x))$, introduce the Wronskian $W(h_{1}\lc h_{k})\in\mathcal{A}((x))$: 
\[W(h_{1}\lc h_{k}) = \det (\partial_{x}^{j-1}h_{i})_{i,j=1}^{n}.\]
It is straightforward to check that 
\begin{equation}\label{Wro}
W(f_{1}\lc f_{i})=w_{0}x^{q}+w_{1}x^{q+1}+\dots ,
\end{equation}
where $q=\sum_{n+1-i}^{n}\nu_{j}$ and $w_{0}\neq 0$. In particular, $W(f_{1}\lc f_{i})\neq 0$ in $\mathcal{A}((x))$ for all $i=1\lc n$. 

Denote 
\[g_{i}=\frac{W(f_{1}\lc f_{i})}{W(f_{1}\lc f_{i-1})}.\]
Then, following the proof of the Proposition 6.2 in \cite{TU1}, one can check that
\[D=\left(\partial-\frac{g'_{n}}{g_{n}}\right)\left(\partial-\frac{g'_{n-1}}{g_{n-1}}\right)\left(\partial-\frac{g'_{1}}{g_{1}}\right), \]
where $g'_{i}=\partial g_{i}$.

Formula \eqref{Wro} implies
\[\frac{g'_{i}}{g_{i}}=\frac{\nu_{n+1-i}}{x}+c_{0}^{(i)}+c_{1}^{(i)}x+\dots , \quad i=1\lc n,  \]
for some $c^{(i)}_{j}\in\mathcal{A}$, so, the lemma is proved. 
\end{proof}

Fix a partition $\lambda=(\lambda_{1},\lambda_{2}\lc \lambda_{l},0,\dots)$, where $\lambda_{l}\neq 0$. We are going to give a definition of a pseudo-differential operator with residue $\lambda$ as well as a pseudo-differential operator without monodromy. These definitions are based on the construction of the ideal $\underline{J}_{t}(\bar{\lambda},\emptyset)$ introduced in Section \ref{6}.

For a pseudo-differential operator $D\in\Psi\mathcal{D}\bigl(\mathcal{A}((x))\bigr)$ of the form
\begin{equation}\label{psD}
D=\partial^{t}+\sum_{i=1}^{\infty}\left(\sum_{j=-i}^{\infty}\beta_{ij}x^{j}\right)\partial^{t-i}, 
\end{equation}
introduce the following polynomials in a variable $\alpha$:
\begin{equation}\label{rho k}
\rho_{k}(\alpha)=\sum_{i=0}^{l+k}\beta_{i,-i+k}(\alpha+l+k)^{\underline{l+k-i}}.
\end{equation}
We will say that $D$ has residue $\lambda$ if 
\begin{align*}
& \beta_{0j}=0,\quad  -i\leq j\leq -l-1, \quad\text{and} \\
& \rho_{0}(\lambda_{i}-i)=0,\quad i=1\lc l.
\end{align*}

Suppose that $D$ has residue $\lambda$, and fix $i,j\in\Z_{\geq 1}$, $i<j$. Denote $K=\lambda_{i}-\lambda_{j}+j-i$. Consider a matrix
\begin{equation}\label{R_ij}
R_{ij}=
\begin{tikzpicture}[baseline=(current bounding box.center)]
\matrix (m) [matrix of math nodes,nodes in empty cells,right delimiter={)},left delimiter={(} ]
{
\rho_{1,1} & \rho_{2,1} & & & \rho_{K,1}\\
\rho_{0,2} & \rho_{1,2} & & & \\
0 & \rho_{0,3} & & &  \\
&  & & & \rho_{2,K-1} \\
0 & & 0\,\,\, & \rho_{0,K} & \rho_{1,K}\\
} ;
\draw[loosely dotted, thick] (m-2-2)-- (m-5-5);
\draw[loosely dotted, thick] (m-3-2)-- (m-5-4);
\draw[loosely dotted, thick] (m-1-2)-- (m-4-5);
\draw[loosely dotted, thick] (m-3-1)-- (m-5-1);
\draw[loosely dotted, thick] (m-5-1)-- (m-5-3);
\draw[loosely dotted, thick] (m-1-2)-- (m-1-5);
\draw[loosely dotted, thick] (m-1-5)-- (m-4-5);
\draw[loosely dotted, thick] (m-3-1)-- (m-5-3);
\end{tikzpicture},
\end{equation}
where $\rho_{k,s}$ are defined as follows: for each $s\in\Z_{\geq 1}$, set $\alpha_{s}=\lambda_{j}-j+s-1$; then 
\begin{equation}\label{rho ks}
\rho_{k,s}=(\alpha_{s}+l+K-s+1)^{\underline{K-s+1}}\rho_{k}(\alpha_{s}).
\end{equation}

For each $s=1\lc K-1$, if $\alpha_{s+1}=\lambda_{d}-d$ for some $d\in\Z_{\geq 1}$, cross out the $s$-th column of $R_{ij}$ and the $(s+1)$-st row of $R_{ij}$, and denote the resulting matrix as $R_{ij}^{red}$. 

We will say that the pseudo-differential operator $D$ has no monodromy if, additionally to having residue $\lambda$, it satisfies the following conditions: 
\begin{align*}
& \beta_{ij}=0,\quad i> 1,\,\, -i\leq j\leq -l-1, \quad\text{and} \\
& \det R_{ij}^{red} =0, \quad i=1\lc l,\,\, j>i.
\end{align*}

Let $\Psi\mathcal{D}(\mathcal{A}((x)),\lambda )$ denote the subset of all pseudo-differential operators of the form \eqref{psD} with residue $\lambda$ and with no monodromy.

For a partition $\lambda=(\lambda_{1}\lc \lambda_{l},0,\dots)$, where $\lambda_{l}\neq 0$, and natural numbers $s$, $q$ such that $s\geq\lambda_{1}$ and $q\geq l$, define partitions $\lambda^{r(s)}$ and $\lambda^{c(q)}$ as follows. The Young diagram of the partition $\lambda^{r(s)}$ is obtained from the Young diagram of the partition $\lambda$ by attaching a row of length $s$ to the top side of the diagram. The Young diagram of the partition $\lambda^{c(q)}$ is obtained from the Young diagram of partition $\lambda$ by attaching a column of length $q$ to the left side of the diagram, see the pictures below.
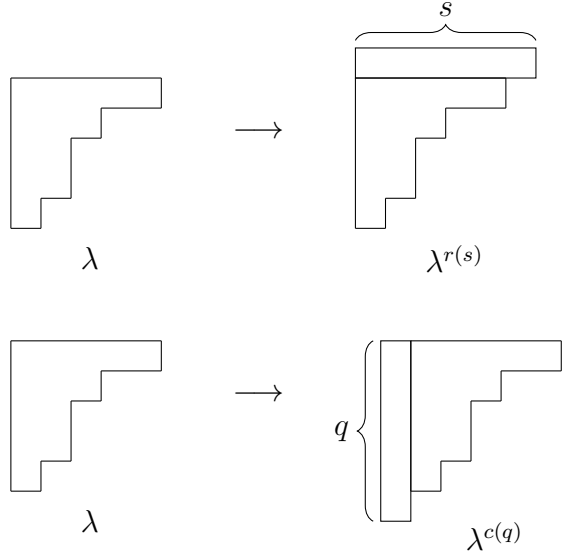
\begin{align*}
& \begin{tikzpicture}[baseline=-1.8cm]
 \draw (0,-1) -- (2,-1);
 \draw (2,-1) -- (2,-1.4);
 \draw (2,-1.4) -- (1.2,-1.4);
 \draw (1.2,-1.4) -- (1.2,-1.8);
 \draw (1.2,-1.8) -- (0.8,-1.8);
 \draw (0.8,-1.8) -- (0.8,-2.6);
 \draw (0.8,-2.6) -- (0.4,-2.6);
 \draw (0.4,-2.6) -- (0.4,-3);
 \draw (0.4,-3) -- (0,-3);
 \draw (0,-1) -- (0,-3);
 \node[text width=4pt] at (1,-3.4) 
    {$\lambda$};
\end{tikzpicture}
\quad\quad\longrightarrow\quad\quad 
\begin{tikzpicture}[baseline=-1.8cm]
 \draw (0,-1) -- (2,-1);
 \draw (2,-1) -- (2,-1.4);
 \draw (2,-1.4) -- (1.2,-1.4);
 \draw (1.2,-1.4) -- (1.2,-1.8);
 \draw (1.2,-1.8) -- (0.8,-1.8);
 \draw (0.8,-1.8) -- (0.8,-2.6);
 \draw (0.8,-2.6) -- (0.4,-2.6);
 \draw (0.4,-2.6) -- (0.4,-3);
 \draw (0.4,-3) -- (0,-3);
 \draw (0,-1) -- (0,-3);
 \draw[black] (0,-0.6) rectangle ++(2.4,-0.4);
 \draw [decorate,
    decoration = {brace, amplitude=6pt}] (0,-0.5) --  (2.4,-0.5) node[pos=0.5,above=5pt]{$s$};
 \node[text width=4pt] at (1,-3.4) 
    {$\lambda^{r(s)}$};   
\end{tikzpicture} \\
 & {} \\
& \begin{tikzpicture}[baseline=-1.8cm]
 \draw (0,-1) -- (2,-1);
 \draw (2,-1) -- (2,-1.4);
 \draw (2,-1.4) -- (1.2,-1.4);
 \draw (1.2,-1.4) -- (1.2,-1.8);
 \draw (1.2,-1.8) -- (0.8,-1.8);
 \draw (0.8,-1.8) -- (0.8,-2.6);
 \draw (0.8,-2.6) -- (0.4,-2.6);
 \draw (0.4,-2.6) -- (0.4,-3);
 \draw (0.4,-3) -- (0,-3);
 \draw (0,-1) -- (0,-3);
 \node[text width=4pt] at (1,-3.4) 
    {$\lambda$};
\end{tikzpicture}
\quad\quad\longrightarrow\quad 
\begin{tikzpicture}[baseline=-1.8cm]
 \draw (0,-1) -- (2,-1);
 \draw (2,-1) -- (2,-1.4);
 \draw (2,-1.4) -- (1.2,-1.4);
 \draw (1.2,-1.4) -- (1.2,-1.8);
 \draw (1.2,-1.8) -- (0.8,-1.8);
 \draw (0.8,-1.8) -- (0.8,-2.6);
 \draw (0.8,-2.6) -- (0.4,-2.6);
 \draw (0.4,-2.6) -- (0.4,-3);
 \draw (0.4,-3) -- (0,-3);
 \draw (0,-1) -- (0,-3);
 \draw[black] (-0.4,-1) rectangle ++(0.4,-2.4);
 \draw [decorate,
    decoration = {brace,mirror,amplitude=6pt}] (-0.5,-1) --  (-0.5,-3.4) node[pos=0.5,left=5pt]{$q$};
 \node[text width=4pt] at (0.8,-3.6) 
    {$\lambda^{c(q)}$};   
\end{tikzpicture}
\end{align*}
\begin{lem}\label{l2}
Fix a partition $\lambda=(\lambda_{1},\lambda_{2},\dots)$ and a natural number $s$ such that $s\geq\lambda_{1}$. Let $D$ be a pseudo-differential operator from $\Psi\mathcal{D}(\mathcal{A}((x)),\lambda)$. Consider a pseudo-differential operator
\[\tilde{D}=\left(\partial-\frac{s}{x}+\sum_{i=0}^{\infty}a_{i}x^{i}\right) D\]
with some $a_{i}\in\mathcal{A}$. Then $\tilde{D}\in \Psi\mathcal{D}(\mathcal{A}((x)),\lambda^{r(s)})$.
\end{lem}
\begin{proof}
Let $\beta_{i}(x)$, $i\in\Z_{\geq 0}$, be the coefficients of $D$:
\[D=\sum_{i=0}^{\infty} \beta_{i}(x)\partial^{t-i}.\]
We have $\beta_{0}(x)=1$, and $\beta_{i}(x)=\sum_{j=-i}^{\infty}\beta_{ij}x^{j}$ for some $\beta_{ij}\in\mathcal{A}$. We will  also put $\beta_{-1}(x)=0$.

Let $\tilde{\beta}_{i}(x)$, $i\in\Z_{\geq 0}$, be the coefficients of $\tilde{D}$:
\[\tilde{D}=\sum_{i=0}^{\infty} \tilde{\beta}_{i}(x)\partial^{t+1-i}.\]
Then 
\[\tilde{\beta}_{i}(x)=\beta_{i}(x)+\left(\partial-\frac{s}{x}+\sum_{i=0}^{\infty}a_{i}x^{i}\right)\beta_{i-1}(x),\]
in particular, $\tilde{\beta}_{i}(x)=\sum_{j=-i}^{\infty}\tilde{\beta}_{ij}x^{j}$ for some $\tilde{\beta}_{ij}\in\mathcal{A}$, and for each $j, k\in\Z_{\geq 0}$,
\begin{equation}\label{l2p1}
\tilde{\beta}_{j,-j+k}=\beta_{j,-j+k} -\beta_{j-1,-j+1+k}(s+j-k-1)
 +\sum_{k'=1}^{k}\beta_{j-1,-j+k'}a_{k-k'}.
\end{equation}
Let $l$ be the number of non-zero elements in the partition $\lambda$. 
Since $D\in\Psi\mathcal{D}(\mathcal{A}((x)),\lambda)$, we have 
\begin{equation}\label{l2p2}
\beta_{ij}=0,\quad  -i\leq j<-l.
\end{equation}
Together with \eqref{l2p1}, this implies 
\begin{equation}\label{l2p3}
\tilde{\beta}_{ij}=0,\quad  -i\leq j<-l-1.
\end{equation}
Let $\rho_{k}(\alpha)$, $\rho_{k,s}$, $R_{ij}$ ($R_{ij}^{red}$) be the objects introduced in formulas \eqref{rho k}, \eqref{rho ks}, \eqref{R_ij}, respectively, corresponding to the differential operator $D$ and partition $\lambda$. We denote by $\tilde{\rho}_{k}(\alpha)$, $\tilde{\rho}_{k,s}$, $\tilde{R}_{ij}$ ($\tilde{R}_{ij}^{red}$) the similar objects for the pseudo-differential operator $\tilde{D}$ and partition $\lambda^{r(s)}$. 

Using \eqref{l2p2} and \eqref{l2p3}, we can write
\begin{equation}\label{l2p4}
    \begin{split}
        \tilde{\rho}_{k}(\alpha) & =\sum_{j=0}^{k+l+1}\tilde{\beta}_{j,-j+k}(\alpha+l+1+k)^{\underline{l+1+k-j}} = \\
        & =\sum_{j=0}^{k+l}\beta_{j.-j+k}(\alpha+l+1+k)^{\underline{l+1+k-j}} - \sum_{j=0}^{k+l}\beta_{j.-j+k}(s+j-k)(\alpha+l+1+k)^{\underline{l+k-j}} + \\
        & + \sum_{j=0}^{k+l}\sum_{k'=1}^{k}a_{k-k'}\beta_{j,-j+k'-1}(\alpha+l+1+k)^{\underline{l+k-j}} = \\
        & = \rho_{k}(\alpha+1)(\alpha+1+k-s)+\sum_{k'=1}^{k}a_{k-k'}(\alpha+1+k+l)^{\underline{k-k'+1}}\rho_{k'-1}(\alpha+1).
    \end{split}
\end{equation}
In particular, for $k=0$, we have 
\begin{equation}\label{l2p5}
\tilde{\rho}_{0}(\alpha)=\rho_{0}(\alpha+1)(\alpha+1-s)
\end{equation}
Since $\rho_{0}(\lambda_{i}-i)=0$, $i=1\lc l$, formula \eqref{l2p5} gives $\tilde{\rho}_{0}(\tilde{\lambda}_{i}-i)=0$, $i=1\lc l+1$, where
\[ \tilde{\lambda}_{i}=
\begin{cases}
    s, \quad i=1, \\
    \lambda_{i-1}, \quad i>1. 
\end{cases}
\]
Therefore, $\tilde{D}$ has residue $\lambda^{r(s)}=(\tilde{\lambda}_{1}\lc\tilde{\lambda}_{l+1},0,\dots )$. 

Put $\lambda_{0}=s$, and for each $j\in\Z_{\geq 1}$, define $R_{0j}$ in the same way we defined $R_{ij}$ for positive $i$.

Fix some $i,j\in\Z_{\geq 1}$, $i<j$, $1\leq i\leq l$, and for each $k=1\lc K$, denote by $\tilde{c}_{k}$, (resp., $c_{k}$) the $k-th$ column in the matrix $\tilde{R}_{ij}$ (resp., $R_{i-1,j-1}$). Then the expression that we obtained in \eqref{l2p4} for $\tilde{\rho}_{k}(\alpha)$ gives 
\begin{equation}\label{l2p6}
    \tilde{c}_{k}=c_{k}(\tilde{\lambda}_{j}-j+1+k-s)+\sum_{k'=1}^{k-1}c_{k-k'}a_{k'-1}.
\end{equation}
 
For each $i'$, $j'$, let us call a column or a row of the matrix $R_{i'j'}$ (resp., $\tilde{R}_{i'j'}$) reduced if it is crossed out when we obtain the matrix $R^{red}_{i'j'}$ (resp., $\tilde{R}^{red}_{i'j'}$). We claim that \eqref{l2p6} implies 
\begin{equation}\label{l2p7}
    \det \bigl(\tilde{R}^{red}_{ij}\bigr) = \dashprod_{k=1}^{K}(\tilde{\lambda}_{j}-j+1+k-s)\det \bigl(R^{red}_{i-1,j-1}\bigr),
\end{equation}
where $K=\tilde{\lambda}_{i}-\tilde{\lambda}_{j}+j-i$, and ${\scalebox{1.2}{$\Pi$}'}_{k=1}^{K}$ means the product over $k$ such that the column $c_{k}$ is not reduced.

Indeed, if $j=i+1$, then $R_{i-1,j-1}=R^{red}_{i-1,j-1}$, $\tilde{R}_{ij}=\tilde{R}^{red}_{ij}$, and \eqref{l2p6} immediately gives \eqref{l2p7}. If $j>i+1$, we should be more accurate since $R_{i-1,j-1}$ and $\tilde{R}_{ij}$ contain some reduced rows. It is easy to check that the numbers of the reduced rows and columns in the matrices $R_{i-1,j-1}$ and $\tilde{R}_{ij}$ coincide. A column $c_{k}$ of the matrix $R_{ij}$ is reduced if and only if $k=\lambda_{j'}-\lambda_{j-1}+j-1-j'$ for some $j'$ such that $i-1<j'<j-1$. But since $\det R^{red}_{i-1,j'}=0$, this column is a linear combination of $c_{k'}$ with $k'<k$ (up to elements of reduced rows of $R_{i-1,j'}$, which are also elements of reduced rows of $R_{i-1,j-1}$). Therefore, in the case $j>i+1$, \eqref{l2p6} implies \eqref{l2p7} as well.

If $i\geq 1$, then $\det \bigl(R^{red}_{i-1,j-1}\bigr)=0$, and \eqref{l2p7} implies $\det \bigl(\tilde{R}^{red}_{ij}\bigr)=0$. 

If $i=1$, consider the factor $\tilde{\lambda}_{j}-j+1+K-s$ in the right hand side of \eqref{l2p7}. Since for $i=1$, we have $K=\tilde{\lambda}_{1}-\tilde{\lambda}_{j}+j-1$, and $\tilde{\lambda}_{1}=s$, this factor equals zero, and we have $\det \bigl(\tilde{R}^{red}_{1j}\bigr)=0$.

Therefore, $\tilde{D}$ has no monodromy and the lemma is proved.
\end{proof}

\begin{lem}\label{l3}
Fix a partition $\lambda = (\lambda_{1}\lc\lambda_{l},0,\dots)$, $\lambda_{l}\neq 0$, and a natural number $q$ such that $q\geq l$. Let $D$ be a pseudo-differential operator from $\Psi\mathcal{D}(\mathcal{A}((x)), \lambda )$. Consider a pseudo-differential operator
\[\tilde{D}=\left(\partial + \frac{q}{x}+\sum_{i=1}^{\infty}a_{i}x^{i}\right)^{-1} D\]
with some $a_{i}\in\mathcal{A}$. Then $\tilde{D}\in\Psi\mathcal{D}(\mathcal{A}((x)), \lambda^{c(q)} )$.
\end{lem}
\begin{proof}
Let again $\beta_{i}(x)=\sum_{i=-j}^{\infty}\beta_{ij}x^{j}$ and $\tilde{\beta}_{i}(x)$ be the coefficients of $D$ and $\tilde{D}$, respectively, like in the proof of the previous lemma; the only difference is that now $\tilde{D}$ has order $t-1$, not $t+1$. We will also assume that $\tilde{\beta}_{-1}(x)=0$.

We  have 
\[\beta_{i}(x)=\tilde{\beta}_{i}(x)+\left(\partial + \frac{q}{x}+\sum_{i=1}^{\infty}a_{i}x^{i}\right)\tilde{\beta}_{i-1}(x).\]
In particular, $\tilde{\beta}_{i}(x)=\sum_{i=-j}^{\infty}\tilde{\beta}_{ij}x^{j}$ for some $\tilde{\beta}_{ij}\in\mathcal{A}$, and for each $j,k\in\Z_{\geq 0}$, 
\begin{equation}\label{l3p1}
\beta_{j,-j+k}=\tilde{\beta}_{j,-j+k}+\tilde{\beta}_{j-1,-j+1+k}(q-j+k+1)+\sum_{k'=1}^{k}\tilde{\beta}_{j-1,-j+k'}a_{k-k'},
\end{equation}

Let us prove by induction on $i$ that 
\begin{equation}\label{l3p4}
\tilde{\beta}_{ij}=0\quad\text{for all } i\in\Z_{\geq 0},\, -i\leq j<-q-1.
\end{equation}
For the base $i=0$, there is nothing to check.
Since $D\in\Psi\mathcal{D}(\mathcal{A}((x)),\lambda)$, we have 
\begin{equation}\label{l3p2}
\beta_{ij}=0,\quad  -i\leq j<-l.
\end{equation}
Therefore, for $-i\leq j<-l$, formula \eqref{l3p1} gives
\begin{equation}\label{l3p3}
\tilde{\beta}_{ij}=-\tilde{\beta}_{i-1,j+1}(q+j+1)+\sum_{j'=-i+1}^{j}\tilde{\beta}_{i-1,j'}a_{j-j'}.
\end{equation}
If $j<-q-2$, then by induction assumption $\tilde{\beta}_{i-1,j+1}=0$ and $\tilde{\beta}_{i-1,j'}=0$ for $j'=-i+1\lc j$. Then \eqref{l3p3} gives $\tilde{\beta}_{ij}=0$.

If $j=-q-1$, then the first term in the right hand side of \eqref{l3p3} is zero because of the factor $q+j+1$ and all other terms are zero because of the induction assumption. Therefore, $\tilde{\beta}_{i,-q-1}=0$, and \eqref{l3p4} follows.

Let $\rho_{k}(\alpha)$, $\rho_{k,s}$, $R_{ij}$ ($R_{ij}^{red}$) be the objects introduced in formulas \eqref{rho k}, \eqref{rho ks}, \eqref{R_ij}, respectively, corresponding to the differential operator $D$ and partition $\lambda$. We denote by $\tilde{\rho}_{k}(\alpha)$, $\tilde{\rho}_{k,s}$, $\tilde{R}_{ij}$ ($\tilde{R}_{ij}^{red}$) the similar objects for the pseudo-differential operator $\tilde{D}$ and partition $\lambda^{c(q)}$. 

Using \eqref{l3p1}, \eqref{l3p4}, and \eqref{l3p2}, we can write
\begin{equation*}
\begin{split}
& \rho_{k}(\alpha)(\alpha+k+q+1)^{\underline{q-l+1}} = \sum_{j=0}^{k+q+1}\beta_{j,-j+k}(\alpha+k+q+1)^{\underline{k+q+1-j}} = \\
& = \sum_{j=0}^{k+q}\tilde{\beta}_{j,-j+k}(\alpha+1+k+q)^{\underline{k+q+1-j}} +\sum_{j=0}^{k+q}\tilde{\beta}_{j,-j+k}(\alpha+1+k+q)^{\underline{k+q-j}}(q-j+k) + \\
& +\sum_{j=0}^{k+q}\sum_{k'=1}^{k}\beta_{j,-j+k'-1}(\alpha+1+k+q)^{\underline{k+q-j}}a_{k-k'} = \\
& = (\alpha+1+q+k)\tilde{\rho}_{k}(\alpha+1)+\sum_{k'=1}^{k}a_{k-k'}(\alpha+1+k+q)^{\underline{k-k'+1}}\tilde{\rho}_{k'-1}(\alpha+1).
\end{split}
\end{equation*}
Therefore, we get
\begin{equation}\label{l3p5}
\rho_{k}(\alpha)(\alpha+k+q)^{\underline{q-l}}=\tilde{\rho}_{k}(\alpha+1)+\sum_{k'=1}^{k}a_{k-k'}(\alpha+k+q)^{\underline{k-k'}}\tilde{\rho}_{k'-1}(\alpha+1).
\end{equation}

In particular, for $k=0$, we have 
\begin{equation}\label{l3p6}
    \tilde{\rho}_{0}(\alpha)=\rho_{0}(\alpha-1)\cdot (\alpha+l)(\alpha+l+1)\dots (\alpha+q-1).
\end{equation}
Since $\rho_{0}(\lambda_{i}-i)=0$ for $i=1\lc l$ , and $\lambda_{i}=0$ for $i>l$, formula \eqref{l3p6} gives $\tilde{\rho}_{0}(\tilde{\lambda}_{i}-i)=0$, $i=1\lc q$, where $\tilde{\lambda}_{i}=\lambda_{i}+1$ for each $i=1\lc q$. Therefore, $\tilde{D}$ has residue $\lambda^{c(q)}=(\tilde{\lambda}_{1}\lc\tilde{\lambda}_{q},0,\dots)$.

Fix some $i,j\in\Z_{\geq 1}$, $i<j$, $1\leq i\leq l$, and denote by $\tilde{c}_{k}$ (resp., $c_{k}$) the $k$-th column of the matrix $\tilde{R}_{ij}$ (resp., $R_{ij}$). Then one can check that \eqref{l3p5} implies 
\begin{equation}\label{l3p7}
c_{k}(\tilde{\lambda}_{j}-j+k+q)^{\underline{q-l+1}}=\tilde{c}_{k}(\tilde{\lambda}_{j}-j+k+q)+\sum_{k'=1}^{k}a_{k-k'}\tilde{c}_{k'-1}.
\end{equation}

For any $i'$, $j'$, let us define the reduced rows and columns of a matrix $R_{i'j'}$ or $\tilde{R}_{i'j'}$ in the same way we did it in the proof of the previous lemma. We will prove by induction on $j$ that 
\begin{equation}\label{l3p8}
    \det \bigl(\tilde{R}^{red}_{ij}\bigr) = \dashprod_{k=1}^{K}(\tilde{\lambda}_{j}-j+k+q-1)^{\underline{q-l}}\det \bigl(R^{red}_{i,j}\bigr),
\end{equation}
where $K=\tilde{\lambda}_{i}-\tilde{\lambda}_{j}+j-i$, and ${\scalebox{1.2}{$\Pi$}'}_{k=1}^{K}$ means the product over $k$ such that the column $c_{k}$ is not reduced.

For $j=i+1$, we have $\tilde{R}^{red}_{ij}=\tilde{R}_{ij}$ and $R_{ij}^{red}=R_{ij}$. Then \eqref{l3p7} implies \eqref{l3p8}. 

Consider the case $j>i+1$. A column $\tilde{c}_{k}$ is reduced if and only if $k=\tilde{\lambda}_{j'}-\tilde{\lambda}_{j}+j-j'$ for some $j'$ such that $i<j'<j$. But by induction assumption, we have $\det \tilde{R}^{red}_{ij'} = 0$, therefore, the column $\tilde{c}_{k}$ is a linear combination of columns $\tilde{c}_{k'}$ with $k<k'$ (up to elements of reduced rows of $\tilde{R}_{ij'}$, which are elements of reduced rows of $\tilde{R}_{ij}$). Therefore, again, \eqref{l3p7} implies \eqref{l3p8}.

If $1\leq i\leq l$, then $\det \bigl(R^{red}_{i,j}\bigr)=0$, and \eqref{l3p8} implies $\det \bigl(\tilde{R}^{red}_{i,j}\bigr)=0$.

If $l+1\leq i\leq q$, then consider the factor $(\tilde{\lambda}_{j}-j+K+q-1)^{\underline{q-l}}$ in the right hand side of \eqref{l3p8}. Since $K=\tilde{\lambda}_{i}-\tilde{\lambda}_{j}+j-i$ and $\tilde{\lambda}_{i}=1$, this factor equals $(q-i)^{\underline{q-l}}$, which is zero for $i>l$. Therefore, we have again $\det \bigl(\tilde{R}^{red}_{i,j}\bigr)=0$.

Hence, $\tilde{D}$ has no monodromy, and the lemma is proved. 
\end{proof}

Recall that for $b(u)\in\mathcal{A}(u)$, $L_{z}\bigl[b(u)\bigr](x)\in\mathcal{A}((x))$ denotes the Laurent series of $b(u)$ at $z$, where we take $u-z=x$, in particular $L_{0}\bigl[b(u^{-1})\bigr](x)$ is the Laurent series of $b(u)$ at infinity. We will say that a pseudo-differential operator $D\in\Psi\mathcal{D}(\mathcal{A}(u))$,
\[D=\sum_{i=0}^{\infty}b_{i}(u)\partial_{u}^{t-i},\]
has regular singularity at $\infty$ if for all $i\in\Z_{\geq 0}$,
\[L_{0}\bigl[b_{i}(u^{-1})\bigr](x)=\sum_{k=i}^{\infty}a_{ik}x^{k}\]
for some $a_{ik}\in\mathcal{A}$. Denote the set of all pseudo-differential operators in $\Psi\mathcal{D}(\mathcal{A}(u))$ with regular singularity at $\infty$ by $\Psi\mathcal{D}_{\infty}(\mathcal{A}(u))$.

\begin{lem}\label{l4}
    If $D_{1}\in\Psi\mathcal{D}_{\infty}(\mathcal{A}(u))$ and $D_{2}\in\Psi\mathcal{D}_{\infty}(\mathcal{A}(u))$, then $D_{1}D_{2}\in\Psi\mathcal{D}_{\infty}(\mathcal{A}(u))$.
\end{lem}
\begin{proof}
Write
\[ D_{1}=\sum_{i=0}^{\infty}a_{i}(u)\partial_{u}^{t-i},\quad D_{2}=\sum_{j=0}^{\infty}\tilde{a}_{j}(u)\partial_{u}^{s-j}.\]
Then 
\[D_{1}D_{2}=\sum_{l=0}^{\infty}\dbtilde{a}_{l}(u)\partial_{u}^{t+s-l},\]
where 
\begin{equation}\label{l4p1}
\dbtilde{a}_{l}(u)=\sum_{i=0}^{l}\sum_{j=0}^{l-i}\alpha_{ij}^{l}(t)a_{i}(u)(\partial_{u}^{l-i-j}\tilde{a}_{j})(u)
\end{equation}
for some complex numbers $\alpha_{ij}^{l}(t)$.

We have $L_{0}\bigl[a_{i}(u^{-1})\bigr](x)=\sum_{k=i}^{\infty}a_{ik}x^{k}$ and $L_{0}\bigl[\tilde{a}_{j}(u^{-1})\bigr](x)=\sum_{k=j}^{\infty}\tilde{a}_{jk}x^{k}$ for some $a_{ik},\tilde{a}_{jk}\in\mathcal{A}$. Then \eqref{l4p1} implies 
\begin{equation}\label{l4p2}
L_{0}\bigl[\dbtilde{a}_{i}(u^{-1})\bigr](x)=\sum_{i=0}^{l}\sum_{j=0}^{l-i}\alpha_{ij}^{l}(t)\left[\sum_{k_{1}=i}^{\infty}a_{ik_{1}}x^{k_{1}}\right]\left[(-x^{2}\partial)^{l-i-j}\sum_{k_{2}=j}^{\infty}\tilde{a}_{jk_{2}}x^{k_{2}}\right].
\end{equation}

Since 
\[(-x^{2}\partial)^{l-i-j}=\sum_{i'=0}^{l-i-j}d_{i'}x^{2(l-i-j)-i'}\partial^{l-i-j-i'}\]
for some complex numbers $d_{i'}$, we have 
\begin{equation}\label{l4p3}
(-x^{2}\partial)^{l-i-j}\sum_{k_{2}=j}^{\infty}\tilde{a}_{jk_{2}}x^{k_{2}}=\sum_{k_{2}=j}^{\infty}\tilde{\alpha}_{ij}^{l,k_{2}}\tilde{a}_{jk_{2}}x^{l+k_{2}-i-j}
\end{equation}
for some complex numbers $\tilde{\alpha}_{ij}^{l,k_{2}}$.

Notice that for $k_{1}\geq i$ and $k_{2}\geq j$, we have $k_{1}+k_{2}+l-i-j\geq l$, therefore, \eqref{l4p2} and \eqref{l4p3} imply
\[L_{0}\bigl[\dbtilde{a}_{l}(u^{-1})\bigr](x)=\sum_{k=l}^{\infty}\dbtilde{a}_{lk}x^{k}\]
for some $\dbtilde{a}_{lk}\in\mathcal{A}$.
\end{proof}
\begin{lem}\label{l5}
    If $D\in\Psi\mathcal{D}_{\infty}(\mathcal{A}(u))$, then $D^{-1}\in\Psi\mathcal{D}_{\infty}(\mathcal{A}(u))$.
\end{lem}
\begin{proof}
    Write 
\[ D=\sum_{i=0}^{\infty}a_{i}(u)\partial_{u}^{t-i},\quad D^{-1}=\sum_{j=0}^{\infty}\tilde{a}_{j}(u)\partial_{u}^{-t-j}\]
for some $a_{i}(u), \tilde{a}_{j}(u)\in\mathcal{A}(u)$, $a_{0}(u), \tilde{a}_{0}(u)\neq 0$.
Then 
\begin{equation}\label{l5p1}
1=D^{-1}D=\sum_{l=0}^{\infty}\sum_{i=0}^{l}\sum_{j=0}^{l-i}\alpha_{ij}^{l}(t)\tilde{a}_{i}(u)(\partial_{u}^{l-j-i}a_{j}(u))\partial_{u}^{-l}
\end{equation}
for some complex numbers $\alpha_{ij}^{l}(t)$. 
Formula \eqref{l5p1} gives 
\begin{equation}\label{l5p2}
    \sum_{i=0}^{l}\sum_{j=0}^{l-i}\alpha_{ij}^{l}(t)\tilde{a}_{i}(u)(\partial_{u}^{l-j-i}a_{j}(u))=0
\end{equation}
for each $l\in\Z_{\geq 1}$ and $a_{0}(u)\tilde{a}_{0}(u)=1$.

Since $D\in\Psi\mathcal{D}_{\infty}(\mathcal{A}(u))$, for each $i\in\Z_{\geq 0}$, we have
\begin{equation}\label{l5p4}
L_{0}\left[a_{i}(u^{-1})\right]= \sum_{k=i}^{\infty}a_{ik}x^{k}
\end{equation}
with some $a_{ik}\in\mathcal{A}$. 
We will prove that 
\begin{equation}\label{l5p3}
    L_{0}\left[\tilde{a}_{i}(u^{-1})\right]= \sum_{k=i}^{\infty}\tilde{a}_{ik}x^{k}
\end{equation}
for some $\tilde{a}_{ik}\in\mathcal{A}$ by induction on $i$.

Relation \eqref{l5p4} for $i=0$ together with $a_{0}(u)\tilde{a}_{0}(u)=1$ imply \eqref{l5p3} for $i=0$. 

Fix some $l\in\Z$. Suppose that \eqref{l5p3} is true for any $i<l$. 
Then \eqref{l5p4} implies 
\begin{equation}\label{l5p5}
    L_{0}\left[\tilde{a}_{l}(u^{-1})\right]L_{0}\left[a_{0}(u^{-1})\right] = \sum_{i=0}^{l-1}\sum_{j=0}^{l-i}\alpha_{ij}^{l}(t)\left[\sum_{k_{1}=i}^{\infty}\tilde{a}_{ik_{1}}x^{k_{1}}\right]\left[(-x^{2}\partial)^{l-i-j}\sum_{k_{2}=j}^{\infty}a_{jk_{2}}x^{k_{2}}\right].
\end{equation}
Then the same calculations as in the proof of Lemma \ref{l4} show that \eqref{l5p5} implies \eqref{l5p3} for $i=l$. 
\end{proof}
\subsection{Proof of Theorem \ref{main2}}
Let us use notations from Section \ref{7.2}.
Set $\mathcal{A}=\bigl(\mathcal{A}_{n}\otimes\mathcal{A}_{n'}\bigr)/\bigl(J_{n}(\bar{\nu})\otimes J_{n'}(\bar{\eta})\bigr)$. 
Applying the projection $\mathcal{A}_{n}\otimes\mathcal{A}_{n'}\rightarrow \mathcal{A}$ to the coefficients of the differential operators $D_{n}$ and $D_{n'}$, one obtains two differential operators in $\mathcal{D}(\mathcal{A}(u))$, which we will denote by $\bar{D}_{n}$ and $\bar{D}_{n'}$, respectively. 

Recall the map $L_{z}:\, \mathcal{D}\bigl(\mathcal{A}(u)\bigr) \rightarrow  \mathcal{D}\bigl(\mathcal{A}((x))\bigr)$, see \eqref{Lz}. This map extends in an obvious way to pseudo-differential operators.
Fix some $a=1\lc m$, and consider the differential operators $L_{z_{a}}(\bar{D}_{n}),L_{z_{a}}(\bar{D}_{n'})\in\mathcal{D}(\mathcal{A}((x)))$. By construction of the ideal $J_{n}(\bar{\nu})$ and Proposition \ref{prop1}, the differential equation $L_{z_{a}}(\bar{D}_{n})f=0$ has $n$ solutions $f_{1}\lc f_{n}$ of the form 
$f_{i}=x^{n_{i}}+d_{i}x^{n_{i}+1}+\dots$, where $n_{n+1-i}=n+\nu^{(a)}_{i}-i$.
Similarly, the differential equation $L_{z_{a}}(\bar{D}_{n'})f=0$ has $n'$ solutions $g_{1}\lc g_{n'}$ of the form 
$g_{i}=x^{n'_{i}}+d'_{i}x^{n'_{i}+1}+\dots$, where $n'_{i}=-\eta^{(a)}_{i}+i-1$.

Therefore, by Lemma \ref{l1}, we have
\begin{equation*}
\begin{split}
L_{z_{a}}(\bar{D}_{n}\bar{D}^{-1}_{n'}) & =\left(\partial_{x}-\frac{\nu_{1}^{(a)}}{x}+a_{1}(x)\right)\dots\left(\partial_{x}-\frac{\nu_{n}^{(a)}}{x}+a_{n}(x)\right)\times \\
& \times\left(\partial_{x}+\frac{\eta_{1}^{(a)}}{x}+\tilde{a}_{1}(x)\right)^{-1}\dots\left(\partial_{x}+\frac{\eta_{n'}^{(a)}}{x}+\tilde{a}_{n'}(x)\right)^{-1} 
\end{split}
\end{equation*}
for some $a_{1}(x)\lc a_{n}(x),\tilde{a}_{1}(x)\lc \tilde{a}_{n'}(x)\in\mathcal{A}[[x]]$.

Notice that $1\in\Psi\mathcal{D}(\mathcal{A}((x)),(0,0,\dots))$. Thus, applying Lemma \ref{l3} $n'$ times and Lemma \ref{l2} $n$ times, we conclude that 
\begin{equation}\label{(A)}
L_{z_{a}}(\bar{D}_{n}\bar{D}^{-1}_{n'})\in\Psi\mathcal{D}(\mathcal{A}((x)),(\lambda^{(a)}),\quad a=1\lc k.
\end{equation}

By the construction of the ideals $J_{n}(\bar{\nu})$ and $J_{n'}(\bar{\eta})$, the differential operators $\bar{D}_{n}$ and $\bar{D}_{n'}$ have regular singularity at infinity. Therefore, by Lemma \ref{l4} and Lemma \ref{l5}, we have 
\begin{equation}\label{(B)}
\bar{D}_{n}\bar{D}_{n'}^{-1}\in\Psi\mathcal{D}_{\infty}(\mathcal{A}(u)).
\end{equation}

By the definition of the ideal $\underline{J}_{n-n'}$ and its image $J_{n\vert n'}$ in $\mathcal{A}_{n\vert n'}$, conditions \eqref{(A)} and \eqref{(B)} are equivalent to the condition that the image of $J_{n\vert n'}$ under the projection $\mathcal{A}_{n\vert n'}\rightarrow \mathcal{A}_{n\vert n'}/\tilde{J}_{n\vert n'}$ is zero. But this means that $J_{n\vert n'}\subset \tilde{J}_{n\vert n'}$, and the theorem is proved.

\bibliographystyle{plain}
\bibliography{mybibliography1}

\end{document}